\documentclass[12pt]{amsart}
\usepackage{amssymb,latexsym,amsmath,amsthm,amscd}
\usepackage{setspace}
\usepackage[active]{srcltx}
\usepackage[all]{xy} \xyoption{arc}
\usepackage[left=3cm,top=2cm,right=3cm,bottom = 2cm]{geometry}
\usepackage{graphicx}
\usepackage[usenames,dvipsnames]{color}
\usepackage{tikz}
\usepackage{array}
\usepackage{todonotes}
\usepackage{charter}
\usepackage{hyperref}
\hypersetup{
  colorlinks   = true, %Colours links instead of ugly boxes
  urlcolor     = blue, %Colour for external hyperlinks
  linkcolor    = blue, %Colour of internal links
  citecolor   = blue %Colour of citations
}

\theoremstyle{plain}
\newtheorem{thm}{Theorem}[section]
\newtheorem{lem}[thm]{Lemma}
\newtheorem{prop}[thm]{Proposition}
\newtheorem{cor}[thm]{Corollary}

\theoremstyle{definition}
\newtheorem{dfn}[thm]{Definition}
\newtheorem{ex}[thm]{Example}

\theoremstyle{remark}
\newtheorem{rmk}[thm]{Remark}

\DeclareMathOperator{\PSp}{PSp}

\DeclareMathOperator{\U}{U}

\DeclareMathOperator{\red}{red}

\DeclareMathOperator{\lcm}{lcm}
\DeclareMathOperator{\Gammahat}{\widehat{\Gamma}}

\newcommand{\cE}{\mathcal{E}}

\newcommand{\cO}{\mathcal{O}}

\newcommand{\bF}{\mathbf{F}}

\newcommand{\frakp}{\mathfrak{p}}

\newcommand{\veps}{\varepsilon}

\DeclareMathOperator{\Gal}{Gal}

\DeclareMathOperator{\Sp}{Sp}

\DeclareMathOperator{\Lie}{Lie}

\DeclareMathOperator{\GL}{GL}

\DeclareMathOperator{\SL}{SL}
\DeclareMathOperator{\PSL}{PSL}

\DeclareMathOperator{\Sym}{Sym}

\DeclareMathOperator{\Ind}{Ind}

\newcommand*{\df}{\mathrel{\vcenter{\baselineskip0.5ex \lineskiplimit0pt
                     \hbox{\scriptsize.}\hbox{\scriptsize.}}} =}

\providecommand{\abs}[1]{\left\lvert#1\right\rvert}

\providecommand{\twomat}[4]{\left(\begin{matrix}#1&#2\\#3&#4\end{matrix}\right)}

\providecommand{\pseries}[2]{#1[\![ #2 ]\!]}

\providecommand{\floor}[1]{{\left\lfloor #1 \right \rfloor}}
\providecommand{\ceil}[1]{{\left\lceil #1 \right \rceil}}

\newcommand{\Zp}{\mathbf{Z}_p}

\newcommand{\QQ}{\mathbf{Q}}
\newcommand{\FF}{\mathbf{F}}
\newcommand{\CC}{\mathbf{C}}
\newcommand{\Qbar}{\overline{\mathbf{Q}}}
\newcommand{\ZZ}{\mathbf{Z}}
\newcommand{\PP}{\mathbf{P}}

\begin{document}
%Title
\title{Families of $\phi$-congruence subgroups of the modular group}

\author{Angelica Babei}
\address{McMaster University}
\email{babeia@mcmaster.ca}
\author{Andrew Fiori}
\address{University of Lethbridge}
\email{andrew.fiori@uleth.ca}
\author{Cameron Franc}
\address{McMaster University}
\email{franc@math.mcmaster.ca}

\thanks{The authors gratefully acknowledge financial support received from NSERC through their respective Discovery Grants, the financial support of the University of Lethbridge, and the use of computational resources made available through WestGrid and Compute Canada.}
\date{}

\begin{abstract}
We introduce and study families of finite index subgroups of the modular group that generalize the congruence subgroups. Such groups, termed $\phi$-congruence subgroups, are obtained by reducing homomorphisms $\phi$ from the modular group into a linear algebraic group modulo integers. In particular, we examine two families of examples, arising on the one hand from a map into a quasi-unipotent group, and on the other hand from maps into symplectic groups of degree four. In the quasi-unipotent case we also provide a detailed discussion of the corresponding modular forms, using the fact that the tower of curves in this case contains the tower of isogenies over the elliptic curve $y^2=x^3-1728$ defined by the commutator subgroup of the modular group.
\end{abstract}
\maketitle

\setcounter{tocdepth}{1}
\tableofcontents

\section{Introduction}
The modular group $\Gamma = \SL_2(\ZZ)$ and its finite index subgroups play a fundamental role throughout mathematics and physics. The most well-known of these subgroups are the congruence subgroups, which have the property that membership in the group can be tested against a finite list of congruence conditions. Equivalently, a subgroup is congruence if it contains one of the principal congruence subgroups
\[
  \Gamma(N) = \{\gamma \in \Gamma \mid \gamma \equiv 1 \pmod{N}\}.
\]
The quotients of the complex upper-half plane by these congruence subgroups define the modular curves, which classify elliptic curves decorated with additional structures related to torsion points on the curves. In this paper we study a generalization of this notion of congruence subgroup to families of noncongruence subgroups.

Most subgroups of $\Gamma$ of finite index are not congruence subgroups. In recent years, it has been shown that they too define natural moduli spaces \cite{Chen}, and their arithmetic has been shown to impinge on problems of classical interest. For example, Chen has shown \cite{Chen20} that the arithmetic of certain noncongruence moduli spaces is related to the description of Markoff triples \cite{BourgainGamburdSarnak}. Now that the unbounded denominators conjecture on noncongruence modular forms has been resolved \cite{CalegariDimitrovTang}, one of the most interesting arithmetic open problems in this area, aside from the Grothendieck-Teichmuller conjecture \cite{Schneps}, is to determine the Eisenstein ideals \cite{DworkVanDerPoorten} associated to noncongruence subgroups, which encode the unbounded denominators phenomenon. Determining the Eisenstein ideals arising from $\SL_2(\FF_p)$-Teichmuller level structures would shed light on the work of \cite{Chen20}. 

Future exploration of phenomena such as this could be aided by having at hand a healthy supply of interesting spaces and groups to study. Thinking along these lines, in this paper we generalize the congruence subgroups in a natural way. The notion depends on a group homomorphism
\[
\phi \colon \Gamma \to G(\Qbar)
\]
where $G$ is some algebraic group defined over a number field. Given such data, there exists a number field $K/\QQ$ such that $\phi$ takes values in $G(\cO_K[1/M])$ for some integer $M\geq 1$, where $\cO_K$ is the ring of integers in $K$. Therefore it makes sense to reduce this homomorphism $\phi$ modulo all but finitely many prime ideals of $K$. The corresponding \emph{principal $\phi$-congruence subgroups} of level $N$ for $\gcd(M,N)=1$ are defined as
\[
  \Gamma(\phi,N) = \{\gamma \in \Gamma \mid \phi(\gamma) \equiv 1 \pmod{N}\}.
\]
A subgroup of $\Gamma$ is said to be \emph{$\phi$-congruence} if it contains some principal $\phi$-congruence subgroup. These are all finite index subgroups of the modular group, and when $\phi$ is the inclusion of $\SL_2(\ZZ)$ into $\SL_2(\QQ)$ we recover the usual congruence subgroups. When $\phi$ has Zariski dense image, these groups are related to the Malcev completion of $\Gamma$ (see \cite{Hain} for a nice discussion), where $\Gamma$ is interpreted as the fundamental group of the moduli stack of elliptic curves.

Below we study in Section \ref{s:families} basic properties of these $\phi$-congruence subgroups and, in particular, introduce some tools that allow us under suitable hypotheses to lift surjective maps $\Gamma \to G(\FF_p)$ to continuous surjections $\widehat \Gamma \to G(\ZZ_p)$, where $\widehat \Gamma$ denotes the profinite completion of $\Gamma$ and $\ZZ_p$ denotes the $p$-adic integers.

In Section \ref{s:unipotent} we study a particular family of $\phi$-congruence subgroups arising from an upper-triangular but not decomposable representation of $\Gamma$ of rank $2$. The modular curves that arise as $\phi$-congruence modular curves are precisely the curves that live beneath the abelian covers of the elliptic curve $y^2 = x^3-1728$. These curves have quite a long history, going back at least to investigations of Poincar\'{e}, though there remain some open problems (for example, describing the modular curves and Eisenstein constants for the genus zero groups $G_p$ introduced in Subsection \ref{ss:Gp}). We describe how to compute equations for the curves defined by $\phi$-congruence subgroups that lie between the commutator $\Gamma'$ and the double-commutator $\Gamma''$, as well as give a detailed discussion of the corresponding modular forms and the Eisenstein constants. This material uses classical results on division polynomials associated to elliptic curves.

Finally, in Section \ref{s:symplectic} we use the moduli space of rank $4$ irreducible representations of $\Gamma$ described by Tuba-Wenzl \cite{TW} to identify some $\phi$-congruence subgroups related to $\Sp_4$. We were led to consider this case in view of the probabilistic results of \cite{LS} which show that it is rarer to find surjective maps from $\Gamma$ onto $\Sp_4(\FF_p)$ than for other linear algebraic groups. We compute genus and dimension formulas for the corresponding modular curves and modular forms, though we say nothing about finding equations for these curves, which are very high degree covers of the moduli space of elliptic curves, nor do we say anything about the corresponding Eisenstein constants.

By combining our symplectic results with the general theory worked out in Section \ref{s:families}, we are able to establish the following result:
\begin{thm}
\label{t:symplecticmain}
Let $x\in\QQ^\times \setminus (\QQ^\times)^2$, and let $\phi \colon \Gamma \to \Sp_4(\QQ)$ be defined by the matrices
  \begin{align*}
\phi \twomat 1101 =& \left(\begin{matrix}
x & 3x^{-1} & 3x & x^{-1} \\
0 & x^{-1}& 2x & x^{-1} \\
0 & 0 & x & x^{-1} \\
0 & 0 & 0 & x^{-1}
                 \end{matrix}\right), & \phi\twomat{0}{-1}10=& \left(\begin{matrix}
0 & 0 & 0 & -x^{-1} \\
0 & 0 & x & 0 \\
0 & -x^{-1} & 0 & 0 \\
x & 0 & 0 & 0
\end{matrix}\right).
\end{align*}
Let $p>7$ be a prime such that $x$ is a primitive root mod $p$, and let $\widehat \Gamma$ denote the profinite completion of $\Gamma$. Then $\phi$ gives rise to a continuous surjection
\[
\widehat \phi \colon \widehat \Gamma \to \Sp_4(\ZZ_p).
\]
Equivalently, the map $\phi$ induces isomorphisms $\Gamma/\Gamma(\phi,p^n) \cong \Sp_4(\ZZ/p^n\ZZ)$ for all $n\geq 0$.
\end{thm}
See Section \ref{s:symplectic} for more technical details on this result, and for proof of the main statements in Theorem \ref{t:symplecticmain}. In particular, we warn the reader that the homomorphism $\phi$ of Theorem \ref{t:symplecticmain} keeps a nonstandard symplectic form invariant, though the corresponding symplectic group is equivalent to that defined by more standard conventions. Briefly, the proof strategy of Theorem \ref{t:symplecticmain} has two key steps: we first establish the theorem mod $p$ by showing that the image of $\phi$ is not contained in any of the maximal subgroups of $\Sp_4(\FF_p)$, using the explicit enumeration and description of these subgroups as detailed in \cite{King}. Next we use the material from Section \ref{ss:lifting} to give a Hensel lifting type argument to deduce surjectivity mod prime powers from the prime case. The difficulty of executing this strategy more generally rests on the complicated nature of the maximal subgroup structure of classes of finite groups of Lie type, cf. \cite{Aschbacher}. To give the flavor of the difficulties that one encounters, even when considering symplectic groups of degree four over general finite fields, one must consider stabilizers of much more complicated symplectic spreads \cite{BallZieve} than those that arise here over $\FF_p$. For other groups the possibilities are presumably even more foreboding.

There remain many interesting open questions about $\phi$-congruence subgroups, aside from constructing nice examples. For example, how does the absolute Galois group $\Gal(\Qbar/\QQ)$ act on the corresponding curves and associated dessins d'enfants? If the prime $p$ is fixed in Theorem \ref{t:symplecticmain} but $x$ is varied, does this describe a single orbit under the action of the Galois group? In a slightly different direction, under what conditions do these families give rise to expander graphs as in the case of modular curves \cite{Lubotzky}? We hope to return to some of these questions in future work.

\subsection{Acknowledgements} The authors thank Alun Stokes for the illustration of the dessin d'enfants in Figure \ref{f:dessin}, as well as Geoff Mason and Benjamin Breen for some helpful discussions.

\subsection{Notation}
\begin{itemize}
\item[---] $\Gamma = \SL_2(\ZZ)$ or $\PSL_2(\ZZ)$, depending on circumstances. When relevant we will recall which choice is currently in effect;
\item[---]  $ S=(\begin{smallmatrix} 0 & -1 \\ 1 & 0\end{smallmatrix})$, $T=(\begin{smallmatrix} 1 & 1 \\ 1 & 0\end{smallmatrix}) $, $R=ST$;
\item[---] if $G$ is a group, then $G'$ denotes its commutator subgroup;
\item[---] if $K/\QQ$ is a number field, then $\cO_K$ denotes the ring of integers in $K$;
\item[---] $\QQ_p$ and $\ZZ_p$ denote the $p$-adic numbers and the $p$-adic integers, respectively;
\item[---] $v_p$ denotes the normalized $p$-adic valuation on $\QQ_p$ defined such that $v_p(p)=1$;
\item[---] $\tau$ denotes a coordinate on the complex upper-half plane;
\item[---] we write $q = e^{2\pi i\tau/6}$ in Section \ref{s:unipotent}. Any other occurrences of $q$ should be interpreted as meaning $q=e^{2\pi i\tau}$.
\end{itemize}

\section{Families of subgroups}
\label{s:families}
\subsection{Definitions}
Let $\Gamma = \SL_2(\ZZ)$, let $G$ be a finite group, and suppose given a homomorphism
\[
  \phi \colon \Gamma \to G.
\]
The kernel of $\phi$ is necessarily of finite index, and we are interested in two basic questions:
\begin{enumerate}
\item[(a)] When is $\phi$ surjective?
\item[(b)] When is $\ker \phi$ noncongruence?
\end{enumerate}
Since most subgroups of finite index in $\Gamma$ are noncongruence, the answer to question (b) is ``almost always'', and in practice it can be tested easily using Wohlfahrt's criterion. For example, one has the following well-known result whose proof demonstrates this principle:
\begin{lem}
  \label{l:noncong}
If $n \geq 5$ and $\phi \colon \Gamma \to S_n$ is surjective, then $\ker \phi$ is noncongruence.
\end{lem}
\begin{proof}
  Assume that $\Gamma(N) \subseteq \ker \phi$ for some $N \geq 2$. This yields an isomorphism
\[
  \frac{\prod_{p^r \mid \mid N}\SL_2(\ZZ/p^r\ZZ)}{K}\stackrel{\cong}{\longrightarrow} S_n
\]
where $K$ is the image of $\ker \phi$ in the quotient $\Gamma/\Gamma(N)\cong\prod_{p^r \mid \mid N}\SL_2(\ZZ/p^r\ZZ)$. Since $S_n$ has a unique nontrivial normal subgroup, we must have that $N =p^r$ is a prime power. Let $H$ be the kernel of the reduction map $\SL_2(\ZZ/p^r\ZZ) \to \SL_2(\FF_p)$ and let $\phi \colon \SL_2(\ZZ/p^r\ZZ) \to S_n$ denote the surjective homomorphism obtained from our congruence hypothesis. Then $\phi(H)$ is a normal subgroup of $S_n$, and so must be $\{1\}$, $A_n$ or $S_n$. But $H$ is a $p$-group, and so the only possibility is $\phi(H)=\{1\}$. Hence $H\subseteq K$ and we obtain a surjection $\SL_2(\FF_p) \to S_n$, which is absurd, since $\PSL_2(\FF_p)$ and $A_n$ belong to members of different families of finite simple groups.
\end{proof}

The answer to question (a) is more delicate as, for example, the maximal subgroup structure of $G$ can be quite complicated, cf. \cite{Aschbacher}. Nevertheless, the results of papers such as \cite{LS} suggest again that question (a) should frequently have a positive answer. Our aim in this paper is to describe some methods for using global homomorphisms $\phi$ into linear algebraic groups to identify and compute with families of finite-index subgroups of the modular group. We begin with some notation and basic properties. 

Let $G$ be a linear algebraic group defined over a number field and let $\phi \colon \Gamma \to G(\Qbar)$ denote a homomorphism. Since $\Gamma$ is finitely generated, there exists a number field $K/\QQ$ with ring of integers $\cO_K$ such that $\phi$ takes values in $G(\cO_K[1/M])$ for some integer $M\geq 1$. That is, we have a map
\[
  \phi \colon \Gamma \to G(\cO_K[1/M]).
\]
It makes sense to reduce $\phi$ mod integers $N$ coprime to $M$, and the kernels of such maps are our analogues of principal congruence subgroups:
\begin{dfn}
  For $\phi$ as above, we define the corresponding \emph{principal $\phi$-congruence subgroups} of level $N\geq 1$ for $\gcd(M,N)=1$ as
  \[
  \Gamma(\phi,N) = \{\gamma \in \Gamma \mid \phi(\gamma) \equiv 1\pmod{N}\}.
\]
\end{dfn}

\begin{dfn}
  More generally, a subgroup $\Lambda \subseteq \Gamma$ is said to be a \emph{$\phi$-congruence subgroup} if there exists $N\geq 1$ with $\Gamma(\phi,N) \subseteq \Lambda$.
\end{dfn}

\begin{rmk}
Since isomorphism of representations corresponds to matrix conjugation, it is clear that the $\phi$-congruence subgroups of $\Gamma$ only depend on $\phi$ up to isomorphism.
\end{rmk}

\begin{rmk}
Notice that $\Gamma(\phi,1) = \Gamma$ for all representations $\phi$. Therefore one does not have, for example, $\Gamma(\phi,1) \cap \Gamma(N) = \Gamma(\phi,N)$.
\end{rmk}

\begin{rmk}
Suppose that $\Lambda \subseteq \Gamma$ is of finite index, and let $\rho$ be a representation of $\Lambda$ defined over $K$. One can define $\rho$-congruence subgroups $\Lambda(\rho,N)$ analogously to above. However, by considering the block description of matrices defining $\Ind_{\Lambda}^\Gamma\rho$, one sees that
\[
  \Gamma(\Ind_{\Lambda}^\Gamma \rho,N) \subseteq \Lambda(\rho,N).
\]
Therefore every $\rho$-congruence subgroup of $\Lambda$ is an $\Ind_{\Lambda}^\Gamma\rho$-congruence subgroup of $\Gamma$ and so, from this perspective, it suffices to consider $\phi$-congruence subgroups of the full modular group $\Gamma$ as defined above.
\end{rmk}

\begin{lem}
\label{l:basicprops}
  The following properties hold:
  \begin{enumerate}
  \item Each $\Gamma(\phi,N)$ is a normal subgroup of finite index in $\Gamma$, and so each $\phi$-congruence subgroup of $\Gamma$ is of finite index in $\Gamma$.
  \item If $N_1 \mid N_2$ then $\Gamma(\phi,N_2)\subseteq \Gamma(\phi,N_1)$.
  \item One has $\Gamma(\phi,N_1)\cap \Gamma(\phi,N_2) = \Gamma(\phi, \lcm(N_1,N_2))$ and
    \[\bigcap_{\gcd(M,N)=1} \Gamma(\phi,N) = \{1\}.\]
\item Let $\phi_1$ and $\phi_2$ denote two representations of $\Gamma$ defined over $K$. Then
\[
 \Gamma(\phi_1,N)\cap \Gamma(\phi_2,N)\subseteq \Gamma(\phi_1\otimes \phi_2,N).
\]
\end{enumerate}
\end{lem}
\begin{proof}
  Properties (2) and (3) follow directly from the definition, and the normality condition in (1) is clear since $\Gamma(\phi,N)$ is the kernel of the composed homomorphism
  \[
  \Gamma \stackrel{\phi}{\longrightarrow} G(\cO_K[1/M]) \stackrel{\red_N}{\longrightarrow} G(\cO_K[1/M]/N\cO_K[1/M]).
\]
The ring $\cO_K[1/M]/N\cO_K[1/M]$ is Artinian, and so $\red_N\circ \phi$ takes values in a finite group. Therefore its kernel $\Gamma(\phi,N)$ is finite index with
\[
  [\Gamma\colon \Gamma(\phi,N)] \leq \abs{G(\cO_K[1/M]/N\cO_K[1/M])}.
\]

Finally, (4) follows easily by considering the Kronecker products of matrices.
\end{proof}

Recall that the character group of $\Gamma$ is cyclic of order $12$, generated by $\chi$, where $\chi ((\begin{smallmatrix} 1 & 1 \\ 0&1\end{smallmatrix})) = e^{2 \pi i/12}$. This representation is the character of the square of the Dedekind $\eta$-function.

\begin{lem}
With $\chi$ as above, 
\[
  \Gamma(\chi^r,N) = \begin{cases}
\Gamma & N=1,\\
\ker \chi^{2r}& N=2,\\
\ker \chi^r & N> 2.
\end{cases}
\]
\end{lem}
\begin{proof}
 By definition
 \[
   \Gamma(\chi^r,N) = \{\gamma \in \Gamma \mid \chi(\gamma)^r \equiv 1 \pmod{N}\}.
 \]
 Now, $\chi(\gamma)^r$ is a twelfth root of unity. When $N=1$ there is nothing to show. If $N=2$ then $\gamma \in \Gamma(\chi^r,2)$ if and only if $\chi(\gamma)^r = \pm 1$, equivalently, if and only if $\chi(\gamma)^{2r}=1$. When $N > 2$ then $\gamma \in \Gamma(\chi^r,N)$ if and only if $\chi(\gamma)^r=1$, since $1$ is the only twelfth root of unity congruent to $1$ mod $N$ in this case. This concludes the proof.
\end{proof}

\begin{rmk}
Since $\chi^r = \chi^{\otimes r}$, the preceding Lemma illustrates that the inclusion in part (4) of Lemma \ref{l:basicprops} is typically a proper inclusion. Though, if a subgroup is both $\phi_1$-congruence and $\phi_2$-congruence, then Lemma \ref{l:basicprops} shows that it is $\phi_1\otimes\phi_2$-congruence.
\end{rmk}

\begin{dfn}
  The \emph{$\phi$-completion} of $\Gamma$ is the inverse limit
  \[
  \Gammahat_\phi \df \varprojlim_N \frac{\Gamma}{\Gamma(\phi,N)}.
  \]
\end{dfn}

Suppose that $G$ satisfies strong approximation. For each prime $\frakp$ of $\cO_K$, let $\cO_{K,\frakp}$ denote the $\frakp$-adic completion of $\cO_K$. Then $\phi$ induces canonical maps
\[
  \widehat \Gamma \longrightarrow \Gammahat_\phi \longrightarrow \prod_{\gcd(\frakp,M)=1}G(\cO_{K,\frakp}),
\]
where $\widehat \Gamma$ denotes the profinite completion of $\Gamma$.  In the next subsection we introduce some basic tools, likely known to experts, that will be used in Section \ref{s:symplectic} to describe explicit surjections
\[
  \widehat{\Gamma} \to \prod_{p \in S} \Sp_4(\ZZ_p)
\]
for certain (conjecturally) infinite sets of primes $S$. See Corollary \ref{c:sp4surject} below for more details.

\begin{rmk}
All modular curves obtained from the $\phi$-congruence groups $\Gamma(\phi,M)$ have moduli interpretations in the sense of \cite{Chen}. One expects the finite groups underlying the Teichmuller level structures of these moduli interpretations to be closely related to the image of $\phi$, though the relationship is not always clear. The proof of Theorem 5.1 in \cite{BER} can be used to make this relationship explicit in concrete examples. The possible discrepancy between the image of $\phi$ and the corresponding Teichmuller level structures should be borne in mind while reading Section \ref{s:unipotent} below, where we study a family of noncongruence groups arising from a homomorphism $\phi$ with metabelian image. In \cite{ChenDeligne} it is shown that the subgroups of $\Gamma$ defined by metabelian Teichmuller level structures are all congruence subgroups of $\Gamma$. Thus, in the examples of Section \ref{s:unipotent}, the groups underlying the Teichmuller level structures certainly differ from the images of the reductions of $\phi$. 
\end{rmk}

\subsection{Lifting surjectivity mod prime powers}
\label{ss:lifting}
When considering families as above, it is natural to ask when the image of a map
\[ \phi_N : \SL_2(\ZZ) \rightarrow G(\ZZ/N\ZZ) \]
is surjective. It is automatic, by the Chinese remainder theorem, that it suffices to check that $\phi_{p^r}$ is surjective for each $p^{r} || N$.
In this section we show that typically, it would suffice to check each $\phi_p$ where $p|N$.

For the remainder of this section we fix a reductive group $G$ defined over $\cO$,
the ring of integers of some number field, and a prime $\frakp | p$ of $\cO$ which is unramified.
We shall denote by $q = \abs{ \cO/\frakp }$ and $\bF_q = \cO/\frakp$.
We shall also fix a subgroup $L \subset G(\mathcal{O})$.

We will denote by $G_{r} = G(\cO/\frakp^r)$ and for $s<r$ by $G_{r,s}$ the kernel of the map $G_{r} \rightarrow G_s$.
Similarly we shall denote by $L_r$ the image of $L$ in $G_r$ and $L_{r,s}$ the kernel of the map $L_r \rightarrow G_s$.

\begin{lem}
For all but finitely many $p$ we have for each $r\ge 1$ that $G_{r+1,r}$ is isomorphic to the Lie algebra, $\Lie(G_{\bF_q})$, over $\bF_q$ and
 the conjugation action of $G_{r+1}$ on the normal subgroup $G_{r+1,r}$ induces an action of $G_1$ on $G_{r+1,r}$ given by 
 the usual adjoint action on $\Lie(G_{\bF_q})$.
 
 Moreover, for almost all $p$, the map $G_{r+2,r} \rightarrow G_{r+2,r}$ given by $x\mapsto x^p$ induces an isomorphism of groups
 $G_{r+1,r} \rightarrow G_{r+2,r+1}$.
 
Consequently, if for $r_0\ge 1$ we have  $L_{r_0+1,r_0} = G_{r_0+1,r_0}$ then for all $r>r_0$ we have $L_{r,r_0} = G_{r,r_0}$.
 \end{lem}
 \begin{proof}
 We recall that $\Lie(G)$ can be identified with the kernel of the map
    \[        G(\cO[\epsilon]/\epsilon^2) \rightarrow  G(\cO[\epsilon]/\epsilon). \]
 So with $G$ embedded in $\GL_n$ we may identify this with the $n$ by $n$ matrices $X$ for which
  \[   1 + \epsilon X  \in  G(\cO[\epsilon]/\epsilon^2). \]
  Similarly, we have that $G_{r+1,r}$ is precisely the $n$ by $n$ integer matrices $X$ for which
 \[            1+p^r X \in G( \cO/\frakp^{r+1} ). \]
The map $\cO[\epsilon]/\epsilon^2 \rightarrow \cO/\frakp^{r+1}$ given by $\epsilon \mapsto p^r$ then induces the map $\Lie(G)\otimes \bF_q \hookrightarrow G_{r+1,r}$,
which is clearly equivariant for the action of conjugation by $G(\bF_q)$.

One can readily verify by case analysis that for almost all $p$ the map is bijective for all semi-simple groups. In particular, in the applications in this paper we only require this for $\Sp_4$, and so this concludes the proof in that case.

Alternatively, for the general case, notice that for almost all $p$ we have that the implicit function theorem allows us to lift elements of $G_{r+1,r}$ to $G_{r+2,r}$, so that the map  $G_{r+2,r} \rightarrow G_{r+2,r}$ given by $x\mapsto x^p$, so:
  \[  (1 + p^rX)^p = 1+p^{r+1}X \pmod{p^{r+2}}  \]
  induces an injection $G_{r+1,r} \hookrightarrow G_{r+2,r+1}$ which is seen to be a group homomorphism. 
  Whenever the implicit function theorem applies we have that $|G_{r+2,r+1}| = q^{{\rm dim}(G)} = | {\rm Lie}(G) \otimes \bF_q|$.  
  If follows that the injection   
\[  \Lie(G) \otimes \bF_q \rightarrow G_{r+1,r} \rightarrow G_{r+2,r+1} \]
is indeed an isomorphism, as are the maps  $G_{r+1,r} \rightarrow G_{r+2,r+1}$.

 It follows that if $L_{r_0+1,r_0} = G_{r_0+1,r_0}$ then $L_{r+1,r} = G_{r+1,r}$ for all $r\ge r_0$ so that the groups $L_{r,r_0}$ and $G_{r,r_0}$. have the same size, hence are isomorphic.
\end{proof}

% \begin{lem}
% For $p\neq 2$ and each $r\ge 1$, or for $p=2$ and each $r\ge 2$,  the function $G_{r+2,r} \rightarrow G_{r+2,r}$ given by $x\mapsto x^p$ induces an isomorphism of groups
% $G_{r+1,r} \rightarrow G_{r+2,r+1}$.
% Consequently, if for $r_0\ge 1$ we have  $L_{r_0+1,r_0} = G_{r_0+1,r_0}$ then for all $r>r_0$ we have $L_{r,r_0} = G_{r,r_0}$.
% \end{lem}
% \begin{proof}
% Using that $\frakp$ is unramified we have that $p$ is a uniformizer for $\cO_{\frakp}$  then using that elements of $G_{r+2,r}$, once it is embedded in $\GL_n$, can be seen as elements of the form
% $1+p^rX$ we then have
%   \[  (1 + p^rX)^p = 1+p^{r+1}X \pmod{p^{r+2}}.  \]
%   In particular the function $G_{r+2,r} \rightarrow G_{r+2,r}$ induces an injective function $G_{r+1,r} \hookrightarrow G_{r+2,r+1}$ which is then easily seen to be a group isomorphism.
%   
% It follows that if $L_{r_0+1,r_0} = G_{r_0+1,r_0}$ then $L_{r+1,r} = G_{r+1,r}$ for all $r\ge r_0$ so that the groups $L_{r,r_0}$ and $G_{r,r_0}$. have the same size, hence are isomorphic.
% \end{proof}

\begin{lem}
If the adjoint action of $G(\bF)$ on the Lie algebra, $\Lie(G_{\bF_q})$, is irreducible and $L_{2,1}\neq \{0\}$, then $L_{2,1} = G_{2,1}$ and hence $L_{r,1} = G_{r,1}$ for all $r>1$.
\end{lem}
 \begin{proof}
 This is a simple application of Burnside's irreducibility criterion \cite[Thm. 3.10]{isaacs}.
  \end{proof}

\begin{rmk}
Note that the adjoint action would typically be irreducible for simple groups, but not for instance for $\GL_n$, where the adjoint action has a summand being the trivial representation.
\end{rmk}

\begin{lem}
Suppose $\ell$ is the maximum nilpotency degree of an element of $\Lie(G)$.
If $p > 2\ell$ and $L_1$ contains a nilpotent element of $G_1$, then $L_{2,1}\neq \{0\}$.
\end{lem}
 \begin{proof}
 Without loss of generality we may suppose that the element $g\in L$ is an element whose image in $L_1$ is nilpotent.
 Then $g$ has the form $1+X$, will satisfy $g^p=1\pmod{p}$ and has nilpotency degree $\ell' \leq \ell$ so that for $i=1,\ldots,\ell'-1$ we have $X^i$  are linearly independent over $\bF_q$, for $i=\ell',\ldots,p-1$ we have $X^i=0\pmod{p}$ and finally 
$ X^p = 0 \pmod{p^2}$.
 Then we have
\[ (1+X)^p = 1+ pX +  \left( \begin{smallmatrix} p \\ 2 \end{smallmatrix} \right) X^2 + \left( \begin{smallmatrix} p \\ 3 \end{smallmatrix} \right) X^3 + \cdots + \left( \begin{smallmatrix} p \\ \ell-1 \end{smallmatrix}  \right) X^{\ell-1} \pmod{p^2}.\qedhere \]
\end{proof}
\begin{rmk}
%The matrix $\left(\begin{smallmatrix} 3 &1\\0&1\end{smallmatrix}\right)$ is nilpotent modulo $2$ but its square is the identity modulo $4$.
The matrix $\left(\begin{smallmatrix} 4&1\\0&4\end{smallmatrix}\right)$ is nilpotent modulo $3$ but its cube is the identity modulo $9$.
\end{rmk}

Combining these we obtain:
\begin{thm}\label{thm:irred}
Suppose $p$ is at least twice the maximum nilpotency degree of $G$ and $\frakp | p$ is such that $G_{2,1} \cong \Lie(G_{\bF_q})$. Suppose also that the adjoint action of $G(\bF)$ on the lie algebra, $\Lie(G_{\bF_q})$, is irreducible (and non-trivial).
Then if $L_1 = G_1$ we have for all $r\ge 0$ that $L_r = G_r$.
\end{thm}

\section{A unipotent family}
\label{s:unipotent}
\subsection{Generalities on the commutator subgroup}

Let $\Gamma=\SL_2(\ZZ)$, and $K = \QQ(\zeta)$ where $\zeta$ is a primitive twelfth root of unity, so that  $\zeta^4-\zeta^2+1=0$. If $p \equiv 1 \pmod{4}$ then $\FF_{p^2}$ contains all of the roots of $x^4-x^2+1$, so at times below we will restrict to such primes. On the other hand, we will sometimes want the image of $\phi_p$ to not be contained in $\SL_2(\FF_p)$, thus we'll insist that $p\equiv 5\pmod{12}$. This ensures that $x^4-x^2+1$ factors into two irreducible quadratics over $\FF_{p}$. 

Begin by defining a global representation $\phi \colon \Gamma \to \SL_2(\cO_K)$ by setting:
\begin{align*}
\phi(T) &= \twomat{\zeta}{0}{0}{\zeta^{-1}}, & \phi(S) &= \twomat {-\zeta^3}{1}{0}{\zeta^3}, & \phi(R) &= \twomat{1-\zeta^2}{\zeta-\zeta^3}0{\zeta^2}.
\end{align*}
This representation is not irreducible, but it is also not reducible into a direct sum of two one-dimensional subrepresentations. Notice that $T^{12} \in \ker \phi$, so that this representation is certainly not injective!
\begin{lem}
  \label{l:image}
  One has
  \[
  \phi(\Gamma) = \left\{\twomat{u}{uv}{0}{u^{-1}} \mid u^{12}=1,~ v \in \ZZ\zeta\oplus \ZZ\zeta^3\right\}.
  \]
\end{lem}
\begin{proof}
   First observe the following identities:
  \begin{align*}
\phi(T^2ST) &= \twomat 1{\zeta}01, &\phi(T^3S) &= \twomat 1{\zeta^3}01.
  \end{align*}
  This, combined with the description of $\phi(T)$, shows that the set on the right of the equality is contained inside $\phi(\Gamma)$. The reverse inclusion is obvious: notice that set on the right is a subgroup, containing $\phi(T)$ and $\phi(S)$, which generate  $\phi(\Gamma)$.
\end{proof}

\begin{lem}
\label{l:pimage}
  Let $p\equiv 5 \pmod{12}$. Then $\phi_p(\Gamma)$ is the following subgroup of $\SL_2(\FF_{p^2})$ of order $12p^2$,
  \[
\phi_p(\Gamma) = \left\{\twomat {u} v0{u^{-1}}\mid u \in \FF_{p^2},~ u^{12}=1,~ v\in\FF_{p^2}\right\}.
  \]
\end{lem}
\begin{proof}
This follows by reducing Lemma \ref{l:image} mod $p$.
\end{proof}

Let $\Gamma' = [\Gamma ,\Gamma]$ denote the commutator subgroup. For details on the group $\Gamma'$ used below without proof, the reader can consult \cite{Newman}. If $U_p$ is the unipotent radical in $\phi_p(\Gamma)$, then $\phi_p(\Gamma)/U_p \cong \ZZ/12\ZZ$ whenever $p\equiv 5 \pmod{12}$, and so $\phi$ induces an isomorphism $\Gamma/\Gamma' \cong \phi_p(\Gamma)/U_p$. It follows that if $p\equiv 5 \pmod{12}$ then $\phi_p^{-1}(U_p) = \Gamma'$.

It is known that $\Gamma'$ is free nonabelian of rank $2$ generated by $A = [S,R]$ and $B=[S,R^2]$, with the identity $[A,B^{-1}]=-T^6$. The group $\Gamma'$ defines a genus one curve via its action on the upper half plane. For each prime $p\equiv 5\pmod{12}$, the group $\ker \phi_p$ is a normal subgroup of $\Gamma'$ with abelian quotient
\[
  \frac{\Gamma'}{\ker\phi_p} \cong \FF_{p^2}.
\]
Now consider $\Gamma''=[\Gamma', \Gamma']$. Notice that $\Gamma'' \subseteq \ker \phi_p \subseteq \Gamma'$ for all primes $p$, and thus in particular
\[
  \Gamma'' \subseteq \bigcap_{p\equiv 5\pmod{12}} \ker \phi_p \subseteq \Gamma'.
\]
If $U$ denotes the unipotent matrices in $\phi(\Gamma)$, then likewise $\phi(\Gamma') = U$. Since  $U$ is abelian, we have $\Gamma''\subseteq \ker \phi$.

\begin{lem}
  We have
  \[
  \ker \phi = \bigcap_{p\equiv 5\pmod{12}} \ker \phi_p = \Gamma''.
  \]
\end{lem}
\begin{proof}
  Since $\Gamma'$ is freely generated by $A$ and $B$, which are of infinite order, it follows that
  \[
  \frac{\Gamma'}{\Gamma''} \cong \ZZ^2.
\]
If $p\equiv 5\pmod{12}$, then the intermediate identity
\[
\frac{\Gamma'}{\ker \phi_p} \cong \FF_{p^2} \cong (\ZZ/p\ZZ)^2
\]
is the reduction mod $p$ map. It follows that $\Gamma'' = \bigcap_{p\equiv 5\pmod{12}} \ker \phi_p$. Since we have seen that $\Gamma'' \subseteq \ker \phi \subseteq \ker \phi_p$ for all $p$, the lemma follows.
\end{proof}

Summarizing all of this, we have shown that if
\[
  L = \left\{\twomat {u}{uv}0{u^{-1}} \in \GL_2(\CC) \mid u^{12}=1,~ v \in \ZZ\zeta \oplus \ZZ\zeta^3\right\},
\]
then there is an exact sequence
\begin{equation}
  \label{eq:exact}
  1 \to \Gamma'' \to \Gamma \stackrel{\phi}{\to} L \to 1.
\end{equation}

The next lemma illustrates that the notion of $\phi$-congruence subgroups is natural and useful.

\begin{lem}
  \label{l:phicong}
The discrete group $L$ satisfies the congruence subgroup property. Therefore, the finite index subgroups $G\subseteq \Gamma$ containing $\Gamma''$ are exactly the $\phi$-congruence subgroups for this indecomposable representation $\phi$.
\end{lem}
\begin{proof}
Let $L'\subseteq L$ be a finite index subgroup, and let $L'_0 = L'\cap U$, which is finite index in $L$ since an intersection of finite index subgroups is of finite index. Now $L'_0 \subseteq U$ is abelian, and since $U\cong \ZZ\zeta\oplus \ZZ\zeta^3$ visibly satisfies the congruence subgroup property, we find that $L'_0$ contains the congruence subgroup $U^N$ for some $N\geq 1$. But then $L'$ also contains $U^N$, so that it is a congruence subgroup. By the exact sequence \eqref{eq:exact}, the inverse images under $\phi$ of the finite index subgroups of $L$ are exactly the finite index subgroups between $\Gamma''$ and $\Gamma$.
\end{proof}

We can describe $\Gamma''$ more explicitly.  Define $\Delta(N)^-$ to be the smallest normal subgroup of $\Gamma$ containing $-T^N$.
\begin{lem}
\label{l:doublecommutator}
  One has $\Gamma''=\Delta(6)^-$.
\end{lem}
\begin{proof}
  This is due to Newman \cite{Newman}. We noted above that $[A,B^{-1}] = -T^6 \in \Gamma''$, so that $\Delta(6)^-\subseteq \Gamma''$. Since $\Gamma'$ is freely generated by $A$ and $B$, it follows from this that $\Gamma'/\Delta(6)^-$ is abelian. Hence also $\Gamma''\subseteq \Delta(6)^-$.
\end{proof}

\begin{lem}
Let $\chi$ be a one-dimensional representation of $\Gamma'$ such that $\ker \chi$ is a congruence subgroup. Then $\Gamma(6) \subseteq \ker\chi$.
\end{lem}
\begin{proof}
Since $\Gamma'' \subset  \ker \chi$ by Lemma \ref{l:doublecommutator} we have $\Delta(6) \subseteq \pm  \ker \chi$ and hence by Wohlfahrt's Theorem, see \cite[Theorem 1]{Wohlfahrt}, we have
that $\Gamma(6) \subset \pm  \ker \chi$. Then because $\Gamma(6) \cap  \pm \ker \chi = \Gamma(6) \cap  \ker \chi$ we conclude $\Gamma(6) \subseteq  \ker \chi$.
%We know that $\Gamma''$ is the normalizer of $-T^6$. In particular, $T^{12} \in \Gamma''$. Since $\Gamma'' \subseteq \ker \chi$, we thus have $T^{12} \in \ker \chi$. 
%Since  $\ker \chi$ is congruence and normal, then we must have $\Gamma(12)\subseteq \ker \chi$  \angie{find reference}.
\end{proof}

The following argument is adapted from an unpublished result of Geoff Mason, which may have been known earlier to experts.
\begin{lem}
\label{l:phicong22}
Let $G\subseteq \Gamma$ be normal, and assume that $\Gamma''\subseteq G\subseteq \Gamma'$. Then either $G = \Gamma''$ or $G$ is a $\phi$-congruence subgroup of $\Gamma$.
\end{lem}
\begin{proof}
One checks that 
\begin{align*}
SAS^{-1} &= A, & SBS^{-1} &= B^{-1},\\
RAR^{-1} &= A^{-1}B, & RBR^{-1} &= A^{-1}.
\end{align*}
In particular, the induced conjugation action of $R$ and $R^2$ on $\Gamma'/\Gamma'' \cong \ZZ^2$ is fixed-point free. We have $G/\Gamma'' \cong 1,\ZZ$ or $\ZZ^2$. In the first case $G=\Gamma''$, while in the latter we have $G$ of finite index in $\Gamma'$, hence it is $\phi$-congruence by Lemma \ref{l:phicong}. To rule out the case $G/\Gamma''\cong \ZZ$ entirely, assume it is so, and notice that the action of the order six element $R$ would be given by multiplication by $\pm 1$. In particular, $R^2$ acts trivially on $G/\Gamma''$ in this case, contradicting its fixed-point freeness. Therefore we can't have $G/\Gamma''\cong \ZZ$, and this concludes the proof.
\end{proof}

\begin{prop}
All but finitely many of the finite-index normal subgroups of $\SL_2(\ZZ)$ of genus one that contain $-1$ are $\phi$-congruence subgroups for the indecomposable representation $\phi$ introduced above. 
\end{prop}
\begin{proof}
  The subgroups $G$ of $\Gamma$ that are finite index, normal, and genus one were classified by Newman \cite{Newman}. All but finitely many have branch schema $(2,3,N)$ and genus
\[
  g = 1 + \frac{[\Gamma \colon G](N-6)}{24N}.
\]
This already shows $\pm T^6 \in G$ if $g=1$, for then we must have $N=6$, and hence since $-1 \in G$ we have $\Delta(6)^- \subseteq G$. Therefore, $\Gamma'' = \Delta(6)^-\subseteq G$ by Lemma \ref{l:doublecommutator}, and so the Proposition follows by Lemma \ref{l:phicong}.
\end{proof}

\subsection{The tower of isogenies}

\subsubsection{The Modular Curve $X(\Gamma')$} To simplify discussions with modular forms in this section, we now let $\Gamma = \PSL_2(\ZZ)$,  which we maintain for the remainder of the section. Note that $\Gamma/\Gamma'$ is cyclic of order $6$. Recall the following facts about the modular curve $X(\Gamma')$:
\begin{enumerate}
\item There are no elliptic points.
\item There is a unique cusp, which has width $6$.
\item The curve $X(\Gamma')$ has genus $1$.
\item The ring of modular forms is $M(\Gamma') = \CC[\eta^4,E_4,E_6]$.
\end{enumerate}
For explanation of the terminology above, see \cite[Chapter 2]{DS}.

Letting $x=E_4/\eta^8$ and $y=E_6/\eta^{12}$ we thus have that
\[ y^2 = x^3-1728 \]
so that the modular curve $X(\Gamma')$ is naturally the genus $1$ curve with the above Weierstrass equation.
Taking the cusp to be the distinguished point at infinity, it is seen to be an elliptic curve with complex multiplication by $\ZZ[\sqrt{-3}]$.

\subsubsection{\'Etale Covers}

If $\rho : Y \rightarrow X(\Gamma')$ is any isogeny of elliptic curves, then as the map $\rho$ is \'etale, we must have $Y\cong X(G)$, for some finite-index subgroup  $\Gamma'' \subset G \subset \Gamma'$. Moreover, $Y$  has no elliptic points and the cusps, which are the kernel of the isogeny, have width $6$. From the Weierstrass uniformization we see that such isogenies are naturally in bijection with lattices in $\ZZ[\sqrt{-3}]$.

Conversely, if $G$ is any finite index subgroup of $\Gamma'$ with $\Gamma'' \subset G \subset \Gamma'$, then $X(G)$ has no elliptic points and, because  the image of $\Delta(6)^-$ inside  $\PSL_2(\ZZ)$ is contained in $G$,  all cusps have width $6$.
 It follows both that  $ X(G)$ is genus $1$, and the cover $ X(G) \rightarrow X(\Gamma')$ is \'etale. In fact, it is simply an isogeny of elliptic curves.

 We thus have a natural bijection between finite index lattices in $\ZZ[\sqrt{-3}]$ and finite index subgroups satisfying $\Gamma'' \subset G \subset \Gamma'$.
 The simplest part of this correspondence is
 \begin{align*}
 \Gamma'(N) &\df \langle A^N, B^N, \Gamma''\rangle & &\longleftrightarrow&& N\ZZ[\sqrt{-3}].
 \end{align*}
 We see that the index of $\Gamma'(N)$ in $\PSL_2(\ZZ)$ is $6N^2$.

\begin{prop}
The groups $\Gamma'(N)$ are non-congruence for $N > 2$.
\end{prop}
\begin{proof}
If $\Gamma'(N)$ were congruence, then by Lemma \ref{l:doublecommutator} and Wohlfahrt's Theorem, see \cite[Theorem 1]{Wohlfahrt}, $\Gamma'(N)$ would contain $\Gamma(6)$,  which we note has index $2^43^2$.
Since $\Gamma'(N)$ has index $6N^2$, this immediately rules out all options for $N$ except $N=1,2$.
\end{proof}

\begin{rmk}
We note that $\Gamma'(2)$ is a congruence subgroup, and that there exists $G$ with $\Gamma'(3) \subsetneq G \subsetneq \Gamma'$ where $G$ is congruence.
\end{rmk}

Since $\Gamma'(N)$ is normal in $\Gamma'$, it follows that $\Gamma'$ acts on $X(\Gamma'(N))$. By our description of these curves and the corresponding tower of isogenies, there exists a map 
\[\rho_N : \Gamma'\rightarrow X(\Gamma'(N))[N],\]
where $X(\Gamma'(N))[N]$ denotes the $N$-torsion of this elliptic curve, such that the natural action of $ \Gamma'$ on $X(\Gamma'(N))$ is given by $P \mapsto P + \rho_N(g)$.

It follows that an affine model for $X(\Gamma'(N))$ is described by the Weierstrass equation
 \[  \tilde{y}^2 = \tilde{x}^3 - 1728, \]
 and the map $X(\Gamma'(N))\rightarrow X(\Gamma')$ can be understood through the multiplication by $N$ map
 \[ [N](\tilde{x},\tilde{y}) = \left( \frac{\phi_N(\tilde{x},\tilde{y})}{\psi_N(\tilde{x},\tilde{y})^2},   \frac{\omega_N(\tilde{x},\tilde{y})}{\psi_N(\tilde{x},\tilde{y})^3} \right), \]
 with $\psi$, $\phi$, and $\omega$ in $\ZZ[x,y]$ division polynomials, whose definitions can be found in a variety of classical sources, such as \cite[Section 3.2]{WashingtonEC}.

 It follows that, where $X(\Gamma')$ is understood in coordinates $x$ and $y$, we have
 \[ x = \frac{\phi_N(\tilde{x},\tilde{y})}{\psi_N(\tilde{x},\tilde{y})^2} \qquad y = \frac{\omega_N(\tilde{x},\tilde{y})}{\psi_N(\tilde{x},\tilde{y})^3} \]
 and this gives us that
  \[ \psi_N(\tilde{x},\tilde{y})^2E_4 = \phi_N(\tilde{x},\tilde{y})\eta^8 ,\qquad \psi_N(\tilde{x},\tilde{y})^3E_6 = \omega_N(\tilde{x},\tilde{y})\eta^{12}. \]
These expressions can be used to solve for $q$-expansions of $\tilde{x}$ and $\tilde{y}$.
  
  Noting that $\psi_N^2$ and $\phi_N$ are actually polynomials in $\tilde{x}$ of degrees respectively $N^2-1$ and $N^2$, of the form
  \[  \psi_N^2(\tilde{x}) = N^2\tilde{x}^{N^2-1} + O(\tilde{x}^{N^2-2}) \qquad  \phi_N(\tilde{x}) = \tilde{x}^{N^2} + O(\tilde{x}^{N^2-1}), \]
 and that $E_4/\eta^8 \in \pseries{\ZZ}{q}$, we observe that the equation 
  \[ M(\tilde{x}) =  \psi_N(\tilde{x})^2E_4 - \phi_N(\tilde{x})\eta^8 = 0 \]
  gives a polynomial satisfied by $\tilde{x}$ over $\pseries{\ZZ}{q}$, where $q=q_6=e^{2\pi i/6}$. These equations can be used to solve for the $q$-series of $\tilde{x}$ and $\tilde{y}$, see below.  
   Note that $\tilde{x}$ and $\tilde{y}$ have poles of order $2$ and $3$ respectively at the point at infinity, equivalently $q=0$, so that the solutions to these equations are given by Laurent series.

\begin{rmk}
For each $N$ there is an isomorphism $X(\Gamma') \cong X(\Gamma'(N))$ and consequently $X(\Gamma'(N))$ has Weierstrass equation $\tilde{y}^2 = \tilde{x}^3 - 1728$. However, $\tilde{x}$ should not be confused with $x=E_4/\eta^8$.
\end{rmk}
  \begin{rmk}
  By comparing divisors we can see that
  \[ \eta^4 = \frac{dy}{x} = \psi_n(\tilde{x},\tilde{y})\frac{d\tilde{y}}{\tilde{x}}^{\otimes n} \]
  \end{rmk}

We next use the preceding material to describe the primes occurring in unbounded denominators of the modular forms for the groups $\Gamma'(N)$. Of course, such existence results follow much more generally by \cite{CalegariDimitrovTang}, though in the very explicit case considered here one can be more precise about the primes that arise in denominators. Note also that since the forms that arise here can be expressed in terms of forms that transform under a character of the congruence group $\Gamma'$, this case was considered even earlier by Kurth-Long in \cite{KurthLong}. Nevertheless, we believe that our approach via isogenies of elliptic curves is of independent interest.
\begin{prop}\label{bounded}
  The $q$ series of $\tilde{x}$ has coefficients in $\QQ$. Moreover, suppose that $\ell$ is a prime not dividing $N$. Then the $q$-series of $\tilde{x}$ is $\ell$-integral.
  \end{prop}
  \begin{proof}
 Set $\hat{x} = q^2\tilde{x}$ and $\widehat{\eta^4} = q^{-1}\eta^4$, so that $\hat{x}$ and $\widehat{\eta^4}$ are both in $\pseries{\overline{\QQ}}{q}$. We have that $\hat{x}$ satisfies the equation over $\pseries{\ZZ}{q}$
\begin{equation}
\label{eq:hl}
\hat{M}(X) = q^{2N^2-2}M\left(\tfrac{1}{q^2}X\right), \end{equation}
whose reduction modulo $q$ is
   \[ \hat{M}(X)  =  N^2X^{N^2-1} - X^{N^2} \pmod{q}, \]
where the variable $X$ above should not be confused with the modular curves $X(\Gamma'(N))$. From this it follows, using that the constant term of $\hat{x}$ is non-vanishing, that
\[\hat{x} = N^{2} \pmod{q}.\]

   To apply Hensel's lemma  to solve equation \eqref{eq:hl} and thereby find the $q$-series of $\hat{x}$, we shall need to consider the derivative
   \[ \hat{M}'(N^2) = (N^2(N^2-1))(N^2)^{N^2-2} - N^2(N^2)^{N^2-1} =   - N^{2N^2-2} \pmod{q}. \]
   The sources of any irrational quantities that might arise while applying Hensel's lemma are the coefficients of the polynomial occurring in equation \eqref{eq:hl}, as well as from the solution of the first step of Hensel's lemma (the solution modulo $q$). The sources of denominators that arise from this procedure are likewise the coefficients, and inverting the constant term, as a series in $q$, of $\hat{M}'(\hat{x})$. We conclude that the coefficients of $\hat{x}$ are rational and, in fact, integral away from primes dividing $N$.
\end{proof}

The following result is a standard part of the theory of isogenies of elliptic curves:
\begin{lem}
  Suppose $p>3$ is prime and write $\psi_p(X) = \sum_{i} a_i X^i$. Then if $p$ is an ordinary prime for $X(\Gamma')$, we have $p\mid a_i$ for $i> (p^2-p)/2$, whereas if $p$ is a supersingular prime, we have $p\mid a_i$ for all $i>0$.
 When $p=2$ or $p=3$ we have that $\psi_p(X) = 0 \pmod{p}$.
  \end{lem}
  \begin{proof}
  For supersingular primes there are no (non-trivial) $p$-torsion points modulo $p$, hence $\psi_p(X)$ is a constant modulo $p$. 
  
  For ordinary primes there are $p-1$ (non-trivial) $p$-torsion points modulo $p$, so the degree of $\psi_p(X)$ modulo $p$ must be divisible by $(p-1)/2$. 
  Since the lead coefficient is divisible by $p$, so two are the first $(p-1)/2$.
  
  The case of $p=2$ and $p=3$ is a direct check.
  But note that in contrast to the case of other primes, these do not give division polynomials modulo $2$ and $3$, since the discriminant of our curve's model is divisible by both $2$ and $3$.
  \end{proof}

\begin{lem}
  Suppose $p>3$ is prime and write $\psi_{p^r}(X) = \sum_{i} a_i^{(p^r)} X^i$. 
  Then we have that $v_p(a_{(p^{2r}-1)/2 - i}^{(p^r)}) =r$ for $i < (p-1)/2$, while $v_p(a_{(p^{2r}-1)/2 - i}^{(p^r)}) > r - \ceil{ 2i/ (p^2-1)}$ if $i\geq (p-1)/2$.
  \end{lem}
  \begin{proof}
  We note, by comparing divisors and lead terms, that we have the recurrence
  \[ \psi_{p^r}(X) = \psi_{p}(x)^{p^{2r-2}}  \psi_{p^{r-1}}\left( \frac{\phi_{p}(X)}{\psi_p(X)^2} \right), \]
  from which we find
  \[ \psi_{p^r}(X) =  \sum_{i=0}^{m}  a_{m - i}^{(p^{r-1})} \psi_{p}(X)^{2i+1}  \phi_{p}(X)^{m - i} \]
  where for convenience $m= \frac{p^{2(r-1)}-1}{2}.$
  Comparing the degree of $\phi_p(X)$, which is $p^2$, to that of $ \psi_{p}(X)$, which is $(p^2-1)/2$, we see that the highest order terms have $p$-adic valuations coming from $ a_{m - i}^{(p^{r-1})}$ and $\psi_{p}(x)^{2i+1}  $ recursively.
  The $p$-adic valuation $v_p(a_{(p^2-p)/2}^{(p^r)}) = r$ decreases by one when $i$ passes $(p-1)/2$ then again each time $i$ passes a multiple of $(p^2-1)/2$.
    \end{proof}

  \begin{lem}\label{lem-pval}
  Suppose  $p>3$  is prime, write $N=p^rM$ where $p\nmid M$ and write $\psi_{N}(X)^2 = \sum_{i} b_i^{(N)} x^i$. 
  Then we have $v_p(b_{(N^2-1) - i}^{(N)}) = 2r$ for $i < (p-1)/2$ and  $v_p(b_{(N^2-1) - i}^{(N)}) > r - \ceil{ 2i/ (p-1)}$ for $i\geq (p-1)/2$.
  \end{lem}
  \begin{proof}
    We note, by comparing divisors and lead terms, that we have the recurrence
    \[ \psi_{N}(X)^2 = \psi_{M}(X)^{2p^{2r}}\psi_{p^r}\left(  \frac{\phi_M(X)}{\psi_M(X)^2} \right)^2.  \]
    As above, we simply look at the source of the high order terms.
    \end{proof}
  \begin{rmk}
  The preceding lower bounds on valuations in the two lemmas above are not tight in most cases.
  \end{rmk}

     \begin{lem}\label{lempsi23}
    The polynomials
    \begin{align*} 
    \hat{\psi}_N(X) &= \frac{1}{12^{N^2-1}}\psi_N^2(12 X ),&    \hat{\phi}_N(X) &= \frac{1}{12^{N^2}}\phi_N(12 X),\end{align*}     have integer coefficients. Moreover, if for $r\ge 1$ we have $2^r \mid\mid N$, then $\frac{1}{2^{2r}} \hat{\psi}_N(X)^2$ is integral and
     \begin{align*}
     \frac{1}{2^{2r}} \hat{\psi}_N(X)^2 &= X^{N^2-4}(X^3-1) \pmod{2},& \hat{\phi}_N(X) &= X^{N^2} \pmod{2}.\end{align*}
     If for $r\ge 1$ we have $3^r \mid\mid N$, then $\frac{1}{3^{2r}} \hat{\psi}_N(X)^2$ is integral and
     \begin{align*} 
     \frac{1}{3^{2r}}\hat{\psi}_{N}(X)^2 &=X^2(X-1)^{N^2-3} \pmod{3},&  \hat{\phi}_N(X) &= (X- 1)^{N^2} \pmod{3}. \end{align*}
    \end{lem}
    \begin{proof}
    Performing the substitutions $x = 12\hat{x}$ and $y =  24\sqrt{3}\hat{y}$ yields the Weierstrass  equation $y^2 = x^3 - 1$. As this curve is isomorphic to our curve, and the division polynomials are uniquely determined up to scaling, we can compare constant terms to conclude that
     \begin{align*}
     \hat{\psi}_N(\hat{x})^2& = \frac{1}{12^{N^2-1}} \psi_N(12\hat{x})^2,&  \hat{\phi}_N(\hat{x}) &= \frac{1}{12^{N^2}} \phi_N(12\hat{x}). \end{align*}

Next, we note that
         \[ \psi_{N}(X)^2 =   \psi_{p^r}(X)^{2M^{2}}\psi_{M}\left(  \frac{\phi_{p^r}(X)}{\psi_{p^r}(X)^2} \right)^2   \] 
     implies that in either case the divisibility claim we are trying to establish, as well as the claim about reduction for $\hat{\psi}_N$, follow from the analogous claims for $\hat{\psi}_{p^r}$ and $\hat{\phi}_{p^r}$.

Similarly, using
     \[ \phi_N(X) = \psi_{p^r}(X)^{2M^2}\phi_{M}\left( \frac{\phi_{p^r}(X)}{\psi_{p^r}(X)^2} \right), \]
we see that the claims about the reduction reduce to those for the case of $\phi_{p^r}$, and using the divisibility claim about $\psi_{p^r}$.
     
To initiate a proof by induction, we note the following identities:
    \begin{align*}
    \hat{\psi_2}(x)^2 &= 4(x^3 - 1),  &\hat{\phi_2}(x) &= x^4 + 2^3 x,\\
    \hat{\psi_3}(x)^2 &= 9(x^8 -2^3 x^5 + 2^{4}x^2), &\hat{\phi_3}(x) &= x^9 + 2^{5}\cdot 3 x^6+  2^{4}\cdot 3 x^3 - 2^{6}.
\end{align*}
The lemma thus holds in these cases.

For the inductive step, consider the recursive identity
    \[ \hat{\psi}_{p^r}(x)^2 = \hat{\psi}_{p}(x)^{2p^{2r-2}}  \hat{\psi}_{p^{r-1}}\left( \frac{\hat{\phi}_{p}(x)}{\hat{\psi}_p(x)^2} \right)^2 \]
and divide both sides by $p^{2r}$. The claim concerning $\hat{\psi}$  follows by induction.
    
     Next, in considering the recursion
    \[ \phi_{p^r}(x) = \psi_{p^{r-1}}(x)^{2p^2} \phi_p\left( \frac{\phi_{p^{r-1}}(x)}{\psi_{p^{r-1}}(x)^2}  \right), \]
    we see that the claim concerning $\hat{\phi}$ follows by induction. This concludes the proof of the Lemma.
   \end{proof}

  \begin{prop}\label{unbounded}
  Suppose  $p>3$ is an odd prime. For any $p\mid N$ the $q$-series for $\tilde{x}$ has unbounded denominators at $p$.
  \end{prop}
  \begin{proof} 
Beginning as in Proposition \ref{bounded}, we now set
\begin{align*}
\breve{q} &=  \frac{q}{N},&  \breve{x} &= \frac{\hat{x}}{N^2}  =  \frac{q^2}{N^2} \tilde{x} =  \breve{q}^2 \tilde{x}  
\end{align*}
and note $\breve{x} = 1 \pmod{\breve{q}}$.
Next set 
\begin{align*}
 \breve{M}(X) &= \frac{1}{N^{2N^2-2}} \hat{M}( N^2 X) \\
& =  \frac{q^{2N^2-2}}{N^{2N^2-2}}M\left(\tfrac{N^2}{q^2}X\right) \\
&=  \breve{q}^{2N^2-2} M\left( \frac{1}{\breve{q}^2} X\right) \\
&=  \breve{q}^{2N^2-2}\psi_N\left( \frac{1}{\breve{q}^2} X \right)^2E_4 -   N^2\breve{q}^{2N^2}\phi_N\left(  \frac{1}{\breve{q}^2} X \right) \widehat{\eta^8}
\end{align*}
Observe that $\breve{M}(X) \in \pseries{\ZZ}{\breve{q}}[X]$. Note also that, as a series in $\breve{q}$, we have
\begin{align*} E_4 &= \widehat{\eta^8} = 1 \pmod{N},& E_4 &= \widehat{\eta^8} = \breve{x} = 1 \pmod{\breve{q}}. \end{align*}
Consequently:
\begin{align*} 
\breve{M}(X) &= N^2X^{N^2-1} - N^2X^{N^2} \pmod{\breve{q}},\\
\breve{M}(X) &=  \breve{q}^{2N^2-2}\psi_N\left( \frac{1}{\breve{q}^2} X \right)^2   \pmod{\breve{N}}.
\end{align*}
We note that if $\breve{x}$ has unbounded denominators at $p$, then we are done, since then $\tilde{x}$ must also have unbounded denominators at $p$.

We next consider the case that $\breve{x}$ has integral coefficients. Then $\breve{x}$ is a non-zero solution to the non-zero polynomial: 
\[  \breve{M}(X) = \breve{q}^{2N^2-2}\psi_N\left( \frac{1}{\breve{q}^2} X \right)^2 \mod{p}.\]
But then $\breve{x}/\breve{q}^2$ is a non-zero solution to the non-zero polynomial $\psi_N(X)^2.$ However, since $\breve{x}/\breve{q}^2 = 1/\breve{q}^2+O(1/\breve{q})$ this is not possible.

Now, suppose for the purpose of contradiction that the $p$-denominators of $\breve{x}$ as a power series in $\breve{q}$ are bounded. Then consider the minimal power $a\ge 1$ so that $\mathring{x} = p^a\breve{x}$  clears them. Now consider the minimal power $b$ so that $\mathring{M}(X) = p^b \breve{M}(X/p^a)$ has integral coefficients. By Lemma \ref{lem-pval} we can see that  $b = aN^2 - 2v_p(N)$ as the minimal value of $b$ must come from the coefficient of $X^{N^2}$ in $\phi_N(X)$.
It follows that $\mathring{x}$ is a non-zero solution to $X^{N^2} = 0 \pmod{p}$, which is a contradiction.
\end{proof}

 \begin{prop}\label{unbounded23}
  Suppose that either $p=2$ and $4 \mid N$, or $p=3$ and $3 \mid N$. Then the $q$-series for $\tilde{x}$ has unbounded denominators at $p$.
  \end{prop}
  \begin{proof}
   Proceed as in Proposition \ref{unbounded} except set
   \begin{align*} \breve{q}& =  \frac{\sqrt{12}q}{N},&\breve{x} &= \frac{\hat{x}}{N^2}  =  \frac{q^2}{N^2} \tilde{x} =  \breve{q}^2 \tilde{x},
\end{align*}
   so that we obtain
   \begin{align*}
 \breve{M}(X) &= \frac{1}{N^{2N^2-2}} \hat{M}( N^2 X) \\
 &=   \frac{1}{12^{N^2-1}}\breve{q}^{2N^2-2}\psi_N\left( \frac{12}{\breve{q}^2} X \right)^2E_4 -   \frac{N^2}{12^{N^2}}\breve{q}^{2N^2}\phi_N\left(  \frac{12}{\breve{q}^2} X \right) \widehat{\eta^8} .
 \end{align*}
By Lemma \ref{lempsi23} we have that $\frac{1}{N^2}\breve{M}(X)$ has integral coefficients, and by Hensel's lemma we find that $\breve{x}$ will have a $\breve{q}$-series with coefficients in $\ZZ[\sqrt{3}]$.

Reducing modulo $2$ or $\sqrt{3}$ respectively and exploiting our understanding of the shape of $\frac{1}{N^2}\breve{M}(X)$ from Lemma \ref{lempsi23} we conclude the series has infinitely many non-zero coefficients. This yields unbounded denominators, where the Eisenstein constant of \cite{DworkVanDerPoorten} has a growth rate on the order of $2^{r-1}$, if $2^r \mid\mid N$, and $3^{r-1/2}$ if $3^r \mid\mid N$.
\end{proof}

\begin{rmk}
We may obtain a bound on the unboundedness of the denominators $p>3$ by considering the alternative rescaling
\begin{align*}
\breve{q} &=  \frac{q}{N^2},& \breve{x} &= \frac{\hat{x}}{N^2}  =  \frac{q^2}{N^2} \tilde{x} =  N^2\breve{q}^2 \tilde{x}
\end{align*}
with
\begin{align*}
 \breve{M}(X) &= \frac{1}{N^{2N^2}} \hat{M}( N^2 X) \\
& =  \frac{q^{2N^2-2}}{N^{2N^2}}M\left(\tfrac{N^2}{q^2}X\right) \\
&=  N^{2N^2-4}\breve{q}^{2N^2-2} M\left( \frac{1}{N^2\breve{q}^2} X\right) \\
&= N^{2N^2-4} \breve{q}^{2N^2-2}\psi_N\left( \frac{1}{N^2\breve{q}^2} X \right)^2E_4 -   N^{2N^2}\breve{q}^{2N^2}\phi_N\left(  \frac{1}{N^2\breve{q}^2} X \right) \widehat{\eta^8}.
\end{align*}
Observe that $\breve{M}(X) \in \pseries{\ZZ}{\breve{q}}[X]$ and since now
\[ 
\breve{M}(X) = X^{N^2-1} - X^{N^2} \pmod{\breve{q}} 
\]
it follows that we have $\breve{x} \in \pseries{\ZZ}{\breve{q}}$. This shows that the $p$-part of the Eisenstein constant in this case is bounded above by $p^{2v_p(N)}$ where $v_p(N)$ is the $p$-adic valuation of $N$. From the proof of Proposition \label{unbounded} one obtains a lower bound of $p^{v_p(N)}$. These bounds can be observed in Table \ref{t:coeffs}.
\end{rmk}

\begin{ex}
  The computations above can be used to compute the $q$-expansions of $\tilde{x}$ for various groups $\Gamma'(N)$. See Table \ref{t:coeffs} on page \pageref{t:coeffs}.
\end{ex}

  \begin{center}
  \begin{table}
 {\tiny
 \renewcommand{\arraystretch}{1.2}
  \begin{tabular}{l|l}  
  $N$ & $\tilde{x}$ \\
  \hline
 2&$4\*q^{-2}
 + 20\*q^4
 - 28\*q^{10}
 + 12\*q^{16}
 + 60\*q^{22}
 - 128\*q^{28}
 + 36\*q^{34}
 + 232\*q^{40}
 - 384\*q^{46}
 + 88\*q^{52}
 + 596\*q^{58}
 - 1012\*q^{64}+\cdots
$\\[4pt]
3&$9\*q^{-2}
 + \frac{40}{3}\*q^4
 - \frac{68}{81}\*q^{10}
 + \frac{3904}{2187}\*q^{16}
 - \frac{166558}{19683}\*q^{22}
 - \frac{22205536}{1594323}\*q^{28}
 + \frac{990558712}{43046721}\*q^{34}
 + \frac{5210061824}{387420489}\*q^{40}
 - \frac{24842012165}{10460353203}\*q^{46}+\cdots
$\\[4pt]
4&$16\*q^{-2}
 + \frac{77}{4}\*q^4
 + \frac{2189}{256}\*q^{10}
 - \frac{123117}{16384}\*q^{16}
 - \frac{17529627}{1048576}\*q^{22}
 - \frac{145441835}{16777216}\*q^{28}
 + \frac{108979782219}{4294967296}\*q^{34}
 + \frac{3710050797161}{137438953472}\*q^{40}+\cdots
$\\[4pt]
5&$25\*q^{-2}
 + \frac{18104}{625}\*q^4
 + \frac{155226332}{9765625}\*q^{10}
 - \frac{2222658420288}{152587890625}\*q^{16}
 - \frac{311093336095872162}{11920928955078125}\*q^{22}
 - \frac{1904353112035085290144}{186264514923095703125}\*q^{28}+\cdots
$\\[4pt]
6&$36\*q^{-2}
 + \frac{124}{3}\*q^4
 + \frac{1924}{81}\*q^{10}
 - \frac{47996}{2187}\*q^{16}
 - \frac{738988}{19683}\*q^{22}
 - \frac{21736576}{1594323}\*q^{28}
 + \frac{2194018564}{43046721}\*q^{34}
 + \frac{24545039912}{387420489}\*q^{40}+\cdots
$\\[4pt]
7&$49\*q^{-2}
 + \frac{942920}{16807}\*q^4
 + \frac{452445233372}{13841287201}\*q^{10}
 - \frac{345081894895333824}{11398895185373143}\*q^{16}
 - \frac{479562567708066938891706}{9387480337647754305649}\*q^{22}+\cdots
$\\[4pt]
8&$64\*q^{-2}
 + \frac{4685}{64}\*q^4
 + \frac{11236493}{262144}\*q^{10}
 - \frac{42660917997}{1073741824}\*q^{16}
 - \frac{293417059904283}{4398046511104}\*q^{22}
 - \frac{105586521517525931}{4503599627370496}\*q^{28}+\cdots
$\\[4pt]
9&$81\*q^{-2}
 + \frac{22504}{243}\*q^4
 + \frac{259896316}{4782969}\*q^{10}
 - \frac{4743322187456}{94143178827}\*q^{16}
 - \frac{52151470071866590}{617673396283947}\*q^{22}
 - \frac{1078242366892782428512}{36472996377170786403}\*q^{28}+\cdots
$\\[4pt]
10&$100\*q^{-2}
 + \frac{71444}{625}\*q^4
 + \frac{655600868}{9765625}\*q^{10}
 - \frac{9499917323508}{152587890625}\*q^{16}
 - \frac{1242573828554492628}{11920928955078125}\*q^{22}
 - \frac{6786637255738163108224}{186264514923095703125}\*q^{28}+\cdots
$\\[4pt]
11&$121\*q^{-2}
 + \frac{2024888}{14641}\*q^4
 + \frac{2107708958684}{25937424601}\*q^{10}
 - \frac{3463068792019818048}{45949729863572161}\*q^{16}
 - \frac{10266710658444673245952698}{81402749386839761113321}\*q^{22}+\cdots
$\\[4pt]
12&$144\*q^{-2}
 + \frac{1975}{12}\*q^4
 + \frac{2005717}{20736}\*q^{10}
 - \frac{3214590959}{35831808}\*q^{16}
 - \frac{3097831691929}{20639121408}\*q^{22}
 - \frac{1401611862206785}{26748301344768}\*q^{28}+\cdots
$\\[4pt]
13&$169\*q^{-2}
 + \frac{5516600}{28561}\*q^4
 + \frac{203466236008364}{1792160394037}\*q^{10}
 - \frac{910908550673413848384}{8650415919381337933}\*q^{16}
 - \frac{95615817946970947497968276298}{542800770374370512771595361}\*q^{22}+\cdots
$\\[4pt]
14&$196\*q^{-2}
 + \frac{3764876}{16807}\*q^4
 + \frac{1822587755684}{13841287201}\*q^{10}
 - \frac{1392203936568518124}{11398895185373143}\*q^{16}
 - \frac{1917815399212467321320772}{9387480337647754305649}\*q^{22}+\cdots
$\\[4pt]
15&$225\*q^{-2}
 + \frac{482152}{1875}\*q^4
 + \frac{119574864508}{791015625}\*q^{10}
 - \frac{46790246868959936}{333709716796875}\*q^{16}
 - \frac{55028168471537575444694}{234639644622802734375}\*q^{22}+\cdots
$\\[4pt]
16&$256\*q^{-2}
 + \frac{299597}{1024}\*q^4
 + \frac{46170272909}{268435456}\*q^{10}
 - \frac{11226288952865517}{70368744177664}\*q^{16}
 - \frac{4922214057285732035355}{18446744073709551616}\*q^{22}+\cdots
$\\[4pt]
17&$289\*q^{-2}
 + \frac{27586040}{83521}\*q^4
 + \frac{391449987308828}{2015993900449}\*q^{10}
 - \frac{8764054816781581205568}{48661191875666868481}\*q^{16}
 - \frac{353813746945457227736313396858}{1174562876521148458974062689}\*q^{22}+\cdots
$\\[4pt]
18&$324\*q^{-2}
 + \frac{89980}{243}\*q^4
 + \frac{1041204868}{4782969}\*q^{10}
 - \frac{19009193715836}{94143178827}\*q^{16}
 - \frac{208595219596141228}{617673396283947}\*q^{22}
 - \frac{4297641595955485224832}{36472996377170786403}\*q^{28}+\cdots
$\\[4pt]
19&$361\*q^{-2}
 + \frac{53766968}{130321}\*q^4
 + \frac{1487098706349404}{6131066257801}\*q^{10}
 - \frac{1232968136587930849221312}{5480386857784802185939}\*q^{16}+\cdots
$\\[4pt]
20&$400\*q^{-2}
 + \frac{1142861}{2500}\*q^4
 + \frac{671890712717}{2500000000}\*q^{10}
 - \frac{623212175028080877}{2500000000000000}\*q^{16}
 - \frac{5211590495147408005654407}{12500000000000000000000}\*q^{22}+\cdots
$\\[4pt]
21&$441\*q^{-2}
 + \frac{25412248}{50421}\*q^4
 + \frac{332200416637564}{1121144263281}\*q^{10}
 - \frac{6851534014130834623808}{24929383770411063741}\*q^{16}
 - \frac{84933525894031113647653282654}{184773775485920747998089267}\*q^{22}+\cdots
$\\[4pt]
22&$484\*q^{-2}
 + \frac{8098580}{14641}\*q^4
 + \frac{8434771718372}{25937424601}\*q^{10}
 - \frac{13860128833739423988}{45949729863572161}\*q^{16}
 - \frac{41066210805183707028581316}{81402749386839761113321}\*q^{22}+\cdots
$\\[4pt]
23&$529\*q^{-2}
 + \frac{169184120}{279841}\*q^4
 + \frac{14724342306524828}{41426511213649}\*q^{10}
 - \frac{2021812632774024349973568}{6132610415680998648961}\*q^{16}+\cdots
$\\[4pt]
24&$576\*q^{-2}
 + \frac{126391}{192}\*q^4
 + \frac{8217691861}{21233664}\*q^{10}
 - \frac{842969124706799}{2348273369088}\*q^{16}
 - \frac{51972516721973470873}{86566749478060032}\*q^{22}+\cdots
$\\[4pt]
25&$625\*q^{-2}
 + \frac{279018104}{390625}\*q^4
 + \frac{40048194998226332}{95367431640625}\*q^{10}
 - \frac{9069037365909518298420288}{23283064365386962890625}\*q^{16}+\cdots
$\\[4pt]
\end{tabular}}
\caption{Fourier coefficients of modular functions $\tilde{x}$ on $\Gamma'(N)$ for small $N$. The case $N=2$ is the only congruence form above.}
  \label{t:coeffs}
\end{table}
\end{center}
\subsection{Description of modular forms}

\subsubsection{Eisenstein Series}

Suppose that the group $G$ is normal in $\Gamma'$ and that  $g_1,\ldots, g_r$ are coset representatives for $\Gamma'/G$. Then representatives for the cusps are given by $g_1 \infty,\ldots, g_r \infty$.
The stabilizer of the cusp $g_i\infty$ is precisely $G_{g_i\infty} = g_i\{ \left( \begin{smallmatrix} 1 & 6n \\ 0 & 1 \end{smallmatrix}\right) \} g_i^{-1}$
Denote by 
\[ E^{(g_i\infty)}_{2k} = \sum_{g\in G_{g_i\infty} \backslash  G}  1|_{2k}(g_i^{-1}g) = \sum_{g\in G_{\infty} \backslash  g_i^{-1} G g_i} 1|_{2k}(gg_i^{-1}) = \sum_{g\in G_{\infty} \backslash  G} 1|_{k}(gg_i^{-1}) = E_{2k}^{\infty}|_{2k}(g_i^{-1}) \]
the weight $2k$ Eisenstein series for the cusp $g_i\infty$.

\begin{prop}
Let $\chi$ be a character of $\Gamma'/G$ and define
\[ E_{2k}^\chi = \sum_i   \chi(g_i)E_{2k}^{g_i\infty}. \]
Then unless $k=1$ and $\chi=1$, the function is a holomorphic modular form of weight $2k$ and character $\chi$. When $k=1$ and $\chi=1$ this is simply a multiple of $E_2$.
\end{prop}
\begin{proof} The result is standard;  the case $k \ge 2$ is addressed in \cite[Section 2.6]{miyake}. When  $k=1$,  the congruence case is considered in \cite[Theorem 7.2.12]{miyake}, and the  noncogruence case can be derived using methods in \cite[Chapters III-IV]{kubota}.
\end{proof}
%\andrew{
%\begin{align*}
% E_{2k}^\chi|_{2k}(g_j)
% &=  \sum_i   \chi(g_i)E_{2k}^{\infty} |_{2k}(g_i^{-1}g_j) \\
% &=   \sum_i   \chi(g_jg_i)E_{2k}^{\infty} |_{2k}(g_i^{-1})  \\
% &=  \chi(g_j) \sum_i  \chi( g_i)E_{2k}^{\infty} |_{2k}(g_i^{-1})  \\
% &=\chi(g_j)   E_{2k}^\chi
% \end{align*}
%}

\begin{rmk}
For $k>1$ the forms $E^{(g_i\infty)}_{2k}$ are themselves holomorphic modular forms, for $k=1$ they are only quasi-modular.
\end{rmk}

\begin{prop}
Let $G=\Gamma'(N)$. Let $\chi$ be any character of $G/\Gamma'$ and write $\chi(g) = \langle T, \rho_N(g) \rangle$ for some $T\in X(G)[N]$ and where the pairing is the Weil pairing.

For $F \in M_2(G,\chi)$, the divisor of $F$ is precisely
\[ \sum_{[N]R = T} R. \]
\end{prop}
\begin{proof}
For $\chi \neq 1$ we use the standard results used to construct the Weil pairing. That is, there exists a function $g$ on $X(G)$ which has divisor
\[ \sum_{[N]R = T} R- \sum_{[N]R=0} R. \]
and this function transforms with character $\langle T, \rho_N(g) \rangle$.
The function $\eta^4g$ is then a basis for the one dimensional space $ M_2(G,\chi)$.

For $\chi=1$ the function $\eta^4$ is a basis of $M_2(G,\chi)$ and has the correct divisor.
\end{proof}

\begin{prop}\label{formofform}
Let $G=\Gamma'(N)$. Then every weight $2k$ modular form has the form
\[ ( P(\tilde{x})+Q(\tilde{x})\tilde{y})\left(\frac{\eta^4}{\psi_N(\tilde{x},\tilde{y})}\right)^k \]
where $P,Q\in \CC[X]$ respectively have degree at most $\floor{kN^2/2}$ and $\floor{(kN^2-3)/2}$.
\end{prop}
\begin{proof}
One readily verifies that the divisor of $\frac{\eta^4}{\psi_N(\tilde{x},\tilde{y})}$ is precisely $N^2[\infty]$, so that the collection of forms given are indeed all holomorphic. A dimension count then verifies that this gives the whole space.
\end{proof}

\begin{prop}\label{boundedpsi}
The Laurent series for $\frac{1}{\psi_N(\tilde{x},\tilde{y})}$ has bounded denominators for all primes $\ell$ with $(\ell,N)=1$.
\end{prop}
\begin{proof}
From the proof of Proposition \ref{bounded} we verify that the Laurent series of $\psi_N^2$ is of the form 
\[ N^{2N^2}q^{-2N^2-2} + O( q^{-2N^2-1} ) \]
from which if follows that inverting $\psi_N$ does not introduce denominators for primes not dividing $N$.
\end{proof}

\begin{prop}
The unbounded denominators for modular forms on $\Gamma'(N)$ are precisely at odd primes dividing $N$, and if $4|N$, also at $2$.
Specifically, there are modular forms with unbounded denominators at each of these primes, and these are the only primes which can contribute unbounded denominators.
\end{prop}
\begin{proof}
That no other primes appear in denominators follows from Propositions \ref{bounded}, \ref{formofform} and \ref{boundedpsi}, and using that from the equation 
\[ \tilde{y}^2 = \tilde{x}^3 - 1728 \]
the denominators for $\tilde{y}$ cannot be worse than those for $\tilde{x}$, except hypothetically at $p=2$. 
However, in the case of $4\nmid N$ then $\tilde{x}$ generates the function field as an extension of a congruence curve, hence $\tilde{y}$ can be expressed as a polynomial in $\tilde{x}$ with coefficients modular functions with bounded denominators.

That unbounded denominators exist follows immediately from \ref{unbounded}, \ref{unbounded23}, \ref{formofform}, and \ref{boundedpsi}.
\end{proof}

\subsubsection{Dimension Formulas}

\begin{prop}
Suppose $\Gamma'' \subset G \subset \Gamma'$. Denote by $S_k$ and $\cE_k$, respectively, the space of cusp forms of weight $k$, and the space generated by Eisenstein series. Then the following facts hold:
\begin{enumerate}
\item The dimension for the space of modular forms $M_{2k}(G)$ is precisely $k[G:\Gamma']$, and, for $k>1$, we have that $S_{2k}(G)$ has dimension precisely $(k-1)[G:\Gamma']$.

\item The dimension for the space of modular forms $M_{2k}(G,\chi)$ is precisely $k$.

\item For $\chi\neq 1$ we have $S_{2k}(G,\chi) = \eta^4 M_{2(k-1)}(G,\chi)$, it is of dimension $k-1$, and $\cE_{2k}(G,\chi)$ has dimension $1$.

\item For $\chi=1$ we have $S_{2k}(G,\chi) = S_{2k}(\Gamma')$. For $k>1$ this has dimension $k-1$, and for $k=1$ this has dimension $1$. The space $\cE_{2}(G,\chi) = \cE_2(\Gamma')$ has dimension $0$, whereas $\cE_{2k}(G,\chi) = \cE_{2k}(\Gamma')$ has dimension $1$ for $k>1$.
\end{enumerate}
\end{prop}
The claims all follow from general principles, see for example \cite{Schoeneberg}.

\begin{prop}\label{prop:incluofspace}
Suppose $\Gamma'' \subset G \subset \Gamma'$ and fix $N$ such that $\Gamma'(N) \subset G$. Let $\chi$ be any character of $\Gamma'/G$, and denote by $\tilde{\chi}$ the character of $\Gamma'/\Gamma'(N)$ induced
by the map
\[ \Gamma'/\Gamma'(N) \rightarrow \Gamma'/G \overset{\chi}\rightarrow \CC^\times. \]
We have
\[ M_{2k}(G,\chi) = M_{2k}(\Gamma'(N),\tilde{\chi}). \]
\end{prop}
\begin{proof}
We have a natural inclusion $M_{2k}(G,\chi)  \rightarrow  M_{2k}(\Gamma'(N),\tilde{\chi})$, and the spaces have the same dimension.
\end{proof}

\begin{prop}
  \label{p:generation1}
Suppose $\Gamma'' \subset G \subset \Gamma'$ and that for every character $\chi$ of $\Gamma'/G$ there exist non-trivial characters $\chi_1$ and $\chi_2$ of $\Gamma'/G$ such that $\chi = \chi_1\chi_2$. Then the graded ring $M(G)=\oplus_{k \in \ZZ} M_k(G)$ is generated in weight $2$.
In particular this holds provided $\abs{\Gamma'/G} > 3$.
\end{prop}
\begin{proof}
For $k>1$ we have $S_{2k}(G,\chi) = \eta^4 M_{2(k-1)}(G,\chi)$ which, using the dimension formulas, reduces the problem to finding at least one element of $M_{2k}(G,\chi)\setminus S_{2k}(G,\chi) $ coming from lower weights.
Letting $k = k_1+k_2$ we see that $E_{2k_1}^{\chi_1}E_{2k_2}^{\chi_2}$ is precisely such an element.
\end{proof}

\begin{rmk}
Notice that the condition on the group of characters in Proposition \ref{p:generation1} is necessary, as indeed $M(\Gamma')$ is not generated in weight $2$. 
\end{rmk}

\subsection{The groups $G_p$}
\label{ss:Gp}
There are many intermediate $\phi$-congruence subgroups that are not contained in $\Gamma'$, and so are not accounted for in the preceding discussion. One way to study modular forms coming from such groups $G$ is to relate forms on $G$ to forms on $G\cap \Gamma'$. For example, by noting that the associated curves are quotients of an elliptic curve by an action of the group $G/(G\cap \Gamma')$, we can already see via the Hurwitz genus formula that these curves would virtually all have genus $0$.

We will not pursue a detailed study of all these groups here, but instead we close this discussion by identifying a family of these genus zero $\phi$-congruence subgroups that are maximal subgroups of $\Gamma$.

For $p\equiv 5\pmod{12}$ a prime let
\begin{equation}
  \label{d:Gp}
  G_p\df \phi^{-1}_p\left\{ \twomat{u}{v}{0}{u^{-1}}\mid u^4=1, v\in \FF_p\right\}.
\end{equation}
Notice that $S,T^3 \in G_p$ for all primes $p$.
\begin{lem}
For each prime $p\equiv 5\pmod{12}$, the group $G_p$ is noncongruence in $\Gamma$ of index $3p$.
\end{lem}
\begin{proof}
  Since $S,T^3\in G_p$, the Wohlfahrt level is $3$. Therefore, $G_p$ is congruence if and only if $\Gamma(3) \subseteq G_p$. To see that this is not so, notice that
  \begin{align*}
    M&\df\twomat{13}{-9}{3}{-2} = T^4S^{-1}T^{-2}S^{-1}TS^{-1} \in \Gamma(3),\\
    \phi(M) &= \twomat{1}{2\zeta^3-2\zeta}{0}{1}.
  \end{align*}
Since $\zeta^3-\zeta$ is a primitive $12$th root of unity, our hypothesis $p\equiv 5\pmod{12}$ ensures that $\zeta^3-\zeta$ does not reduce mod $p$ to an element of $\FF_p$. We deduce that $M \not \in G_p$, which shows that $\Gamma(3)$ is not contained in $G_p$.
\end{proof}

\begin{lem}
Let $p\equiv 5 \pmod{12}$. The congruence closure of $G_p$ is the unique normal subgroup $\Gamma_{(3)}$ of $\Gamma$ of index $3$.
\end{lem}
\begin{proof}
For this, let
\[
  L_p =\left\{ \twomat{u}{v}{0}{u^{-1}}\mid u^4=1, v\in \FF_{p^2}\right\}.
\]
This is a normal subgroup of $\phi_p(\Gamma)$ of index $3$. Therefore, $\phi^{-1}_p(L_p)$ must be the unique normal subgroup $\Gamma_{(3)}$ of $\Gamma$ of index $3$, which is congruence of level $3$. This shows that $G_p \subseteq \Gamma_{(3)}$. Since the index of $G_p$ inside $\Gamma$ is $3p$, and $p$ is prime, we see that the congruence closure of $G_p$ must be $\Gamma_{(3)}$ as claimed.
\end{proof}

We now briefly illustrate how to study modular forms on these intermediate groups by passing to $\Gamma'$ and using the preceding material. Let $\tilde{G}_p = \Gamma' \cap G_p$ then $X(\tilde{G}_p)$ is a genus $1$ elliptic curve with a $p$-isogeny to $X(\Gamma')$. So $X(\tilde{G}_p)$ has $p$ cusps, the kernel of the isogeny, and no elliptic points.
The group $\tilde{G}_p$ is index $2$ in $G_p$. The group $G_p/\tilde{G}_p$ acts on $X(\tilde{G}_p)$ by the hyper-elliptic involution and hence $X(G_p)$ is the quotient by this involution.

We conclude that $X(G_p)$ has $(p+1)/2$ cusps, the origin being the only $p$-torsion point fixed by the involution. Additionally, $X(G_p)$ has $3$ elliptic points of order $2$, these are the images of the non-trivial $2$-torsion points on $X(\tilde{G}_p)$. Finally, $X(G_p)$ has no elliptic points of order $3$.
The Hurwitz genus formula allows us to see that the genus is $0$.

The dimension formula for the spaces of modular forms of weight $2k$ on $X(\tilde{G}_p)$ is $kp$, whereas on $X(G_p)$ it is
\[ kp/2 + 1 - \begin{cases} \frac{3}{2} & 2 \nmid k  \\ 0 & \text{otherwise}. \end{cases} \]
Comparing with the dimension formula for $X(\Gamma_{(3)})$, where recall that $\Gamma_{(3)}$ is the unique normal subgroup of $\Gamma$ of index $3$,
\[ k/2 + 1 - \begin{cases} \frac{3}{2} & 2 \nmid k  \\ 0 & \text{otherwise} \end{cases} \]
we see that the dimension for the space of forms on $X(G_p)$ is essentially half of that on $X(\tilde{G}_p)$ 
with rounding issues relating to forms on $X(\tilde{G}_p)$ with trivial character, that is, forms in the subring $\CC[\eta^8,E_4,E_6]$. As the modular forms on $X(G_p)$ are those on $X(\tilde{G}_p)$ that are fixed by the hyperelliptic involution, we have in the notation of Proposition \ref{formofform}, and either using Proposition \ref{prop:incluofspace}, or adapting to $X(\tilde{G}_p)$ rather than $X(\Gamma'(p))$, the forms on $X(G_p)$ are of the form
\[ P(\tilde{x}) \left(\frac{\eta^4}{\Psi_p(\tilde{x},\tilde{y})} \right)^k.\]
Finally, noting that the hyperelliptic involution acts on characters by $\chi\mapsto \chi^{-1}$,
a natural basis for forms on $X(G_p)$ comes from adding together forms from $M_k(\tilde{G}_p,\chi)$ and $M_k(\tilde{G}_p,\chi^{-1})$ for each character $\chi$ of $\Gamma'/\tilde{G}_p$.

Though we focused on the specific case of $G_p$ above, the analysis is similar if the image of $G$ in $\Gamma/\Gamma'$ has order $2$ and $\Gamma'(N) \subset G$ with $N$ odd (when $N$ is even the group $G$ need not include the hyperelliptic involution). When the image has elements of order $3$, there is also a quotient by a degree three automorphism of the curve. In this case the analysis is simpler if $3\nmid n$, in which case the fixed points will be three of the $3$-torsion points on the elliptic curve, which leads to two elliptic points.

We conclude this section by illustrating in Figure \ref{f:dessin} the Belyi map corresponding to the smallest genus zero example $G_p$ for $p=5$, which is of index $15$ inside $\Gamma$. Descriptions on how to illustrate Belyi maps via dessins d'enfants are found in a variety of sources, such as \cite{JonesWolfart}.

\begin{figure}
\begin{tikzpicture}
\tikzset{
    bl/.style={circle,minimum size=0.1cm,fill=black,draw=black},
    wh/.style={circle,minimum size=0.1cm,fill=white, draw=black}}

\node[bl] (1) at (-5,0) {};
\node[wh] (4) at (-4,0) {};
\node[wh] (2) at (-3.5,1.5) {};
\node[wh] (3) at (-3.5,-1.5) {};

\draw (1)  to [out=90,in=180,looseness=1]  (2);
\draw (1) -- (4);
\draw (1) to  [out=270,in=180,looseness=1] (3);
\node[bl] (5)  at (-2,0) {};
\draw (2) to  [out=0,in=90,looseness=1] (5);
\draw (3) to  [out=0,in=270,looseness=1] (5);

\node[wh] (6) at (0,0) {};
\draw (5) -- (6);
\node[bl] (7)  at (2,0) {};
\draw (6) -- (7);

\node[wh] (8)   at (3.5,1.5) {};
\node[wh] (9)  at (3.5,-1.5) {};

\draw (7) to  [out=90,in=180,looseness=1] (8);
\draw (7) to  [out=270,in=180,looseness=1] (9);

\node[bl] (10)   at (5,0) {};
\draw (10)  to [out=90,in=0,looseness=1]  (8);
\draw (10)  to [out=270,in=0,looseness=1]  (9);

\node[wh] (11)   at (4.25,0) {};
\draw (10) -- (11);
\node[bl] (12)   at (3.5,0) {};
\draw (11) -- (12);

\node[wh] (13)   at (2.75,0.5) {};
\node[wh] (14)   at (2.75,-0.5) {};

\draw (12) -- (13);
\draw (12) -- (14);
 
\end{tikzpicture}
\caption{The dessin for $G_5$ with index $[\Gamma: G_5]=15$.}
\label{f:dessin}
\end{figure}
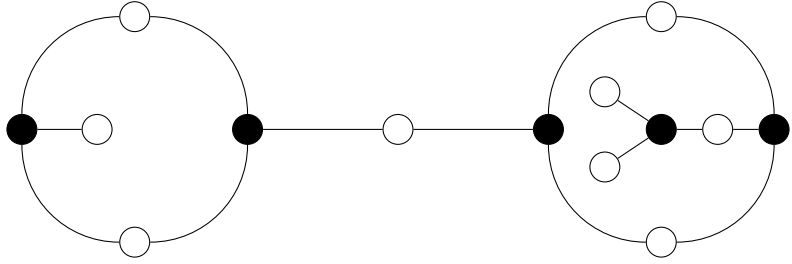

\section{Some symplectic families}
\label{s:symplectic}
The results of \cite{LS} imply that surjective homomorphisms $\Gamma \to \Sp_4(\FF_q)$ are somewhat rarer than for other classical groups. In this section we study families of subgroups arising from such homomorphisms obtained via the moduli results of \cite{TW} on irreducible representations of $\Gamma$ of rank $4$. Throughout this section we take $\Gamma = \SL_2(\ZZ)$, and we write $\rho$ in place of $\phi$ for our homomorphism, to distinguish this case from the previous one.

\subsection{Moduli of representations of rank $4$}
In \cite{TW}, equations for the moduli space of four-dimensional irreducible representations of $\Gamma$ were determined. The corresponding vector-valued modular forms were studied in \cite{FM}. Our goal now is to study which representations in this moduli space admit a symplectic structure. Let us recall from Proposition 2.6 of \cite{TW} that the moduli space of irreducible representations of $\Gamma$ has irreducible components determined by matrices
\begin{align*}
\rho(T) &= \left(\begin{matrix}
x&(1+D^{-1}+D^{-2})y&(1+D^{-1}+D^{-2})z&w\\
0&y&(1+D^{-1})z&w\\
0&0&z&w\\
0&0&0&w
\end{matrix}\right),\\
\rho(S) &= \left(\begin{matrix}
0&0&0&-(x^2yD^3)^{-1}\\
0&0&D(y^2z)^{-1}&0\\
0&-D^2(yz^2)^{-1}&0&0\\
(xw^2)^{-1}&&&
\end{matrix}\right),
\end{align*}
for nonzero complex numbers $x,y,z,w$, and where $D = \sqrt{yz/xw}$ denotes a choice of square-root. Moreover, the Corollary above 2.10 of \cite{TW} shows that necessarily $(xyzw)^3=1$, so that the space of irreducible representations has $6$ irreducible components, corresponding to the third root of unity $xyzw$ and the choice of sign in $D$. See Section 2.10 of \cite{TW} for polynomial conditions that the eigenvalues of $\rho(T)$ must satisfy in order to ensure that $\rho$ is irreducible -- since we will be interested in generic specializations of $\rho$, we will not be overly concerned with these conditions. Notice that $\rho$ is defined over a number field if and only if the eigenvalues of $\rho(T)$ are algebraic.

In order to simplify some of the commutative algebra to follow, we write $x = X^2$, $y=Y^2$, $z=Z^2$, $w=W^2$, so that $D= \pm \tfrac{YZ}{XW}$. Since we do not intend to be exhaustive below, we consider only the case $D =\tfrac{YZ}{XW}$ and omit discussion of the case $D=-\tfrac{YZ}{XW}$.

To begin, in terms of our new variables, we have
\begin{align*}
\rho(T) &= \left(\begin{array}{rrrr}
X^{2} & \frac{Y^{2} Z^{2} + X Y Z W + X^{2} W^{2}}{Z^{2}} & \frac{Y^{2} Z^{2} + X Y Z W + X^{2} W^{2}}{Y^{2}} & W^{2} \\
0 & Y^{2} & \frac{Y Z^{2} + X Z W}{Y} & W^{2} \\
0 & 0 & Z^{2} & W^{2} \\
0 & 0 & 0 & W^{2}
                 \end{array}\right),\\
  \rho(S) &= \left(\begin{array}{rrrr}
0 & 0 & 0 & \frac{-W}{X Y^{3} Z^{3}} \\
0 & 0 & \frac{1}{X Y^{3} Z W} & 0 \\
0 & \frac{-1}{X^{2} Z^{2} W^{2}} & 0 & 0 \\
\frac{1}{X^{2} W^{4}} & 0 & 0 & 0
\end{array}\right),
\end{align*}
where $(XYZW)^6=1$. Initially we look for an antisymmetric matrix
\[
  J = \left(\begin{matrix}
      0 & A & B & C\\
      -A & 0 & D & E\\
      -B & -D & 0 &F\\
      -C & -E & -F & 0
    \end{matrix}
  \right)
\]
such that $\rho(T)^tJ\rho(T) = J$ and $\rho(S)^tJ\rho(S)=J$. Using that $XYZW\neq 0$, this already reduces us to considering a matrix of the form:
\begin{align*}
J=&\left(\begin{array}{rrrr}
0 & \frac{-X^{2} W^{2} + 1}{X^{2} W^{2}}&0 & 1 \\
\frac{X^{2} W^{2} - 1}{X^{2} W^{2}}& 0 & \frac{-X Y^{2} Z^{2} W - Y Z - X W}{X Y^{2} Z^{2} W}& 0 \\
0 & \frac{X Y^{2} Z^{2} W + Y Z + X W}{X Y^{2} Z^{2} W}& 0 & -X^{4} Z^{2} W^{6} + X^{2} Z^{2} W^{4}\\
-1 & 0 & X^{4} Z^{2} W^{6} - X^{2} Z^{2} W^{4} & 0
         \end{array}\right)C+\\
  &\left(\begin{array}{rrrr}
0 & \frac{1}{X^{3} Y^{3} Z W^{5}}& \frac{-1}{X^{3} Y^{3} Z W^{5}}& 0\\
\frac{-1}{X^{3} Y^{3} Z W^{5}}& 0 &\frac{-X^{3} Y^{3} Z W^{5} + X^{3} Y Z W^{3} + 1}{X^{3} Y^{3} Z W^{5}}& 1 \\
\frac{1}{X^{3} Y^{3} Z W^{5}}& \frac{X^{3} Y^{3} Z W^{5} - X^{3} Y Z W^{3} - 1}{X^{3} Y^{3} Z W^{5}} & 0 & \frac{X Z W}{Y^{3}}\\
0& -1 & \frac{-X Z W}{Y^{3}} & 0
\end{array}\right)E
\end{align*}

Again, since we do not intend to be exhaustive, we shall concentrate on the case $E=0$. In this case, we can extract some polynomial equations that encode the condition that $J$ defines a symplectic form for $\rho$. These equations describe a variety with eight components, four that are one-dimensional, and four that are two-dimensional. The two-dimensional components are given by the ideals:
\begin{align*}
  I_1 &= (XW+1,YZ+1),\\
  I_2 &= (XW+1,Y^2Z^2-YZ+1),\\
  I_3 &= (XW-1,Y^2Z^2+YZ+1),\\
  I_4 &= (XW-1,YZ-1).
\end{align*}

In what follows we consider only the case $I_4$, where $E=0$, and we set $C=1$, which we are free to do since we may rescale $J$. Under these hypotheses, we can write things easily in terms of our original variables and we find that we are considering the following surface inside of the Tuba-Wenzl moduli space:
\begin{align*}
\rho(T) =& \left(\begin{array}{rrrr}
x & 3 y & 3y^{-1} & x^{-1} \\
0 & y& 2y^{-1} & x^{-1} \\
0 & 0 & y^{-1} & x^{-1} \\
0 & 0 & 0 & x^{-1}
                 \end{array}\right), &\rho(S) =& \left(\begin{array}{rrrr}
0 & 0 & 0 & -x^{-1} \\
0 & 0 & y^{-1} & 0 \\
0 & -y & 0 & 0 \\
x & 0 & 0 & 0
\end{array}\right).
\end{align*}
This family of representations preserves the symplectic form
\[
  J = \left(\begin{array}{rrrr}
0 & 0 & 0 & 1 \\
0 & 0 & -3 & 0 \\
0 & 3 & 0 & 0 \\
-1 & 0 & 0 & 0
\end{array}\right).
\]

\subsection{Combinatorics of some $\Sp_4$ families}
\label{ss:grassmanian}

We now have a generous supply of homomorphisms
\[
\rho \colon \Gamma \to \Sp_4(\QQ).
\]
Reducing modulo $p$ for all but finitely many primes $p$ --- we must avoid $3$ where $J$ has bad reducibility properties, as well as the divisors of $xy$ --- we obtain homomorphisms
\[
  \rho_p \colon \Gamma \to \Sp_4(\FF_p).
\]

As we will see in Theorem \ref{t:sp4surject}, by choosing $x$ and $y$ carefully, this homomorphism is surjective for (conjecturally\footnote{Obtaining surjectivity for infinitely many primes using our approach depends on the validity of Artin's primitive root conjecture, which is currently only known in full generality as a consequence of the Riemann hypothesis for Dirichlet $L$-functions.}) infinitely many primes $p$. We thereby obtain noncongruence subgroups $\ker \rho_p$ of $\Gamma$ of finite index:
\[
  [\Gamma \colon \ker\rho_p] = \abs{\Sp_4(\FF_p)} = p^{4}(p^4-1)(p^2-1),
\] which grows like $O(p^{10})$. This index is too large and poses a problem for explicit study. To circumvent this problem, we compose $\rho_p$ with a primitive action of $\Sp_4(\FF_p)$ on the set of Lagrangian subspaces in $\FF_p^4$, which will pick out a maximal noncongruence subgroup of $\Gamma$ as a point stabilizer. These symplectic point stabilizers are the analogues in our setting of the classical congruence groups $\Gamma_0(p)$.

Assume $p \geq 5$ (later we will require $p\geq 11$) and let our symplectic form be given by $J$ as above. Let $X$ denote the corresponding symplectic Grassmanian. If $p > 3$ is prime, then either $3$ or $-3$ is a quadratic residue mod $p$. Using this, one can show that $X$ is isomorphic mod $p$ to the usual symplectic Grassmanian, and so
\[
  \abs{X(\FF_p)} = (p^2+1)(p+1),
\]
which grows instead like $O(p^3)$. Point stabilizers under this action are known to be isomorphic with $\GL_2(\FF_p)$ semi-direct producted with the additive group $\Sym_2(\FF_p)$ of symmetric $2\times 2$-matrices. An explicit computation with the symplectic form defined by $J$ gives:
\[
  X = \left\{\left \langle\left(\begin{matrix}
          1\\
          0\\
          a\\
          b
        \end{matrix}
      \right),\left(\begin{matrix}
          0\\
          1\\
          c\\
          -3a
        \end{matrix}\right) \right\rangle \right\}\cup \left\{\left \langle\left(\begin{matrix}
          1\\
          a\\
          0\\
          b
        \end{matrix}
      \right),\left(\begin{matrix}
          0\\
          0\\
          1\\
          3a
        \end{matrix}\right) \right\rangle \right\}\cup\left\{\left \langle\left(\begin{matrix}
          0\\
          1\\
          a\\
          0
        \end{matrix}
      \right),\left(\begin{matrix}
          0\\
          0\\
          0\\
          1
        \end{matrix}\right) \right\rangle \right\}\cup \left\{\left \langle\left(\begin{matrix}
          0\\
          0\\
          1\\
          0
        \end{matrix}
      \right),\left(\begin{matrix}
          0\\
          0\\
          0\\
          1
        \end{matrix}\right) \right\rangle \right\}
\]
Let $A(a,b,c)$, $B(a,b)$, $C(a)$ and $D$ denote these Lagrangian subspaces.  We wish to describe how $\rho(S), \rho(R)$ and $\rho(T)$ act on $X$.

\subsubsection{Action of $S$} In this case one easily checks that
\[
  \rho(S)A(a,b,c) = \begin{cases}
    A\left(-\frac{axy}{3a^2+bc},-\frac{cx^2}{3a^2+bc},-\frac{by^2}{3a^2+bc}\right) & 3a^2+bc \neq 0,\\
    B\left(-\frac{ax}{by},-\frac{x^2}{b}\right)& b\neq 0,~ c=-\frac{3a^2}{b},\\
    C\left(-\frac{y^2}{c}\right) & a=b=0,~ c\neq 0,\\
    D& a=b=c=0.
  \end{cases}
\]
Similarly one easily checks:
\[
  \rho(S)B(a,b) =\begin{cases}
    A\left(\frac{axy}{b},-\frac{x^2}{b},\frac{3a^2y^2}{b}\right) & b\neq 0,\\
    B\left(-\frac{x}{3ay},0\right) & a\neq 0,~ b=0,\\
    C(0) & a=b=0.
  \end{cases}
\]
and finally:
\[
  \rho(S)C(a) =\begin{cases} 
    A\left(0,0,-\frac{y^2}{a}\right) & a\neq 0,\\
    B(0,0) & a=0,\\
  \end{cases}
\]
and $\rho(S)D=A(0,0,0)$.

\begin{lem}
  \label{l:sgrassmanian}
Let $\veps_2$ denote the number of fixed points of $\rho(S)$ in its action on the symplectic Grassmanian $X(\FF_p)$. Then
\[
  \veps_2 =  p+2+\left(\frac{-1}{p}\right)=\begin{cases}
p+3 & p\equiv 1 \pmod{4},\\
p+1 & p\equiv 3\pmod{4}.  
\end{cases}
\]
In particular, $\veps_2$ is independent of $x$ and $y$.
\end{lem}
\begin{proof}
By the computation above, the fixed points can only possibly come from $B(a,0)$ with $a\neq 0$ or $A(a,b,c)$ with $3a^2+bc \neq 0$. Notice that $B(a,0)$ is fixed by $\rho(S)$ if and only if $a^2 \equiv -\frac{x}{3y}\pmod{p}$. If this is true then we obtain two elliptic points from these values of $a$. We can write this uniformly in terms of the Legendre symbol as a contribution of $1+\left(\frac{-3xy}{p}\right)$ fixed points.

Now we consider $A(a,b,c)$ for $3a^2+bc\neq 0$. If $a\neq 0$, then $A(a,b,c)$ is fixed if $xy=-3a^2-bc$ and $b=cx/y$. This means, in particular, that $a^2 = -\frac{x(y^2+c^2)}{3y}$. Therefore, in this case we obtain a contribution of:
\begin{align*}
  \sum_{\substack{c=1\\ c^2\not \equiv -y^2\pmod{p}}}^{p}\left(1+\left(\frac{-3xy(y^2+c^2)}{p}\right)\right)=& p-\left(1+\left(\frac{-1}{p}\right)\right)+\left(\frac{-3xy}{p}\right)\sum_{c=1}^{p}\left(\frac{y^2+c^2}{p}\right)\\
  =&p-1-\left(\frac{-1}{p}\right)-\left(\frac{-3xy}{p}\right)
\end{align*}
fixed points.

If $a=0$, then $A(a,b,c)$ is fixed if $b^2=-x^2$ and $c^2=-y^2$. So if $p\equiv 1 \pmod{4}$ this contributes $4$ fixed points, while if $p\equiv 3 \pmod{4}$ it contributes none. This can be written uniformly as $2+2\left(\frac{-1}{p}\right)$. This concludes the proof.
\end{proof}

\subsubsection{Action of $R$} In this case one easily checks that
\begin{align*}
 & \rho(R)A(a,b,c)=\\
 & \begin{cases}
    A\left(-\frac{xy(axy+by^2+3a^2+bc)}{3a^2+bc},\frac{x^2(6axy+3by^2+6a^2+2bc-cx^2)}{3a^2+bc},-\frac{y^2(by^2+6a^2+2bc)}{3a^2+bc}\right) & 3a^2+bc \neq 0,\\
    B\left(-\frac{x(ax+by)}{by^2},\frac{x^2(6a^2x^2+6abxy+2b^2y^2-bx^2y^2)}{b^2y^2}\right)& b\neq 0,~ c=-\frac{3a^2}{b},\\
    C\left(-\frac{y^4+2cy^2}{c}\right) & a=b=0,~ c\neq 0,\\
    D& a=b=c=0.
  \end{cases}
\end{align*}
Similarly one easily checks:
\[
  \rho(R)B(a,b) =\begin{cases}
    A\left(\frac{xy(3a^2y^2+axy-b)}{b},\frac{x^2(2b-9a^2y^2-6axy-x^2)}{b},\frac{y^2(3a^2y^2-2b)}{b}\right) & b\neq 0,\\
    B\left(-\frac{x(3ay+x)}{3ay^2},\frac{x^2(18a^2y^2+18axy+6x^2)}{9a^2y^2}\right) & a\neq 0,~ b=0,\\
    C(-2y^2) & a=b=0,
  \end{cases}
\]
and finally:
\[
  \rho(R)C(a) =\begin{cases} 
    A\left(-\frac{xy(y^2+a)}{a},\frac{x^2(3y^2+2a)}{a},-\frac{y^2(y^2+2a)}{a}\right) & a\neq 0,\\
    B\left(-\frac{x}{y},2x^2\right) & a=0,
  \end{cases}
\]
and $\rho(R)D=A(-xy,2x^2,-2y^2)$.

One could use these explicit formulas to work out the number of elliptic points of order $3$ as in the case of elliptic points of order $2$, but it is somewhat messier than in the preceding case. Therefore we defer this computation until Lemma \ref{l:rgrassmanian} below, when we are prepared to give a simpler representation theoretic proof.

\subsubsection{Action of $T$} Recall that the cusps of the point stabilizers correspond to the cycles in the action of $T$ on $X(\FF_p)$. However, unlike with $S$ and $R$, the order and cycle types of $T$  depend on $x$ and $y$, and the explicit calculations are not as easy or insightful. However, we can still extract useful information   as follows.  Let $\chi$ be the character of the permutation representation above, so that $\chi(g)$ is the number of fixed points of $\rho(g)$ acting on $X(\FF_p)$. The  number of cusps of width $n$ equals the number cycles of length $n$ in the cycle structure of $\rho(T^n)$. In particular, $c_1=\chi(T)$, and by induction for $n>1$, we get the formula 
\[
  c_n = \frac{1}{n}\left(\chi(T^n)-\sum_{\substack{d\mid n\\ d < n}}dc_d\right).
\]
Equivalently,  $\chi(T^n)=\sum_{d\mid n}dc_d$. Performing Mobius inversion yields
\[
  c_n = \frac{1}{n}\sum_{d\mid n} \mu(n/d)\chi(T^d).
\]
The following result is well-known, but we include the proof for completeness.
\begin{lem}
\label{l:cusps}
Let $c$ denote the number of cusps, and let $N$ denote the order of the mod $p$ reduction of $\rho(T)$. Then
  \[
  c =\frac{1}{N}\sum_{d\mid N}\varphi(N/d)\chi(T^d),
\]
where $\varphi$ denotes Euler's totient function.
\end{lem}
\begin{proof}
  Since $N$ is the least common multiple of the cycle lengths of $\rho(T)$ acting on $X$, and these cycle lengths are the widths of the cusps, we have $c=\sum_{n \mid N} c_n$ and therefore
  \begin{align*}
    c=&\sum_{n\mid N}\frac{1}{n}\sum_{d\mid n}\mu(n/d)\chi(T^d)\\
    =&\sum_{d\mid N}\sum_{n\mid (N/d)}\frac{1}{dn}\mu(n)\chi(T^d)\\
    =&\sum_{d\mid N}\frac{1}{d}\chi(T^d)\sum_{n\mid (N/d)}\frac{\mu(n)}{n}\\
    =&\frac{1}{N}\sum_{d\mid N}\varphi(N/d)\chi(T^d)
  \end{align*}
\end{proof}

Thus, to compute the number of cusps, and therefore also the genus, it remains to understand the order of $T$ mod $p$, as well as the character $\chi$. We start this by imposing constraints on $p$, $x$ and $y$, which forces $\rho_p$ to be surjective, as follows.

\begin{thm}
\label{t:sp4surject}
  Suppose that
  \begin{enumerate}
  \item $p > 7$;
  \item $x$ reduces mod $p$ to a generator of $\FF_p^{\times}$;
  \item $x\equiv y^{-1} \pmod{p}$.
  \end{enumerate}
Then $\rho_p(T)$ has order $p(p-1)$, and $\langle \rho_p(T),\rho_p(S)\rangle = \Sp_4(\FF_p)$.
\end{thm}
\begin{proof}
  Let us begin by working over the function field $\QQ(x)$, and write $t = \rho(T)$, $s=\rho(S)$, so that
  \begin{align*}
t =& \left(\begin{smallmatrix}
x & 3x^{-1} & 3x & x^{-1} \\
0 & x^{-1}& 2x & x^{-1} \\
0 & 0 & x & x^{-1} \\
0 & 0 & 0 & x^{-1}
                 \end{smallmatrix}\right), &s =& \left(\begin{smallmatrix}
0 & 0 & 0 & -x^{-1} \\
0 & 0 & x & 0 \\
0 & -x^{-1} & 0 & 0 \\
x & 0 & 0 & 0
\end{smallmatrix}\right).
  \end{align*}

  Let
  \[
P = \left(\begin{smallmatrix}
\frac{3(x^{4} - 1)}{2x(x^{4} + x^{2} + 1)} & 1 & \frac{3(x^{2} + 1)}{2x} & 0 \\
\frac{- x^{6} + x^{4} + x^{2} - 1}{2x(x^{4} + x^{2} + 1)} & 0 & 0 & 1 \\
0 & \frac{-(x^2-1)^2}{2(x^{4} + x^{2} + 1)} & 0 & \frac{x^{2} - 1}{2x^{2}} \\
0 & \frac{(x^{2} - 1)^3}{2(x^{4} + x^{2} + 1)} & 0 & 0
\end{smallmatrix}\right).
\]
and observe that
\[
  \det(P) = \left(\frac{3}{16}\right) \cdot x^{-4} \cdot (x - 1)^{6} \cdot (x + 1)^{6} \cdot (x^{2} - x + 1)^{-2} \cdot (x^{2} + x + 1)^{-2} \cdot (x^{2} + 1)^{2}.
\]
Therefore, since $p\geq 11$ and $x$ generates $\FF_p^\times$, it follows that $\det(P)$ is invertible mod $p$. Set $t_1 = P^{-1}tP$ and observe that
\begin{align*}
t_1 &= \left(\begin{smallmatrix}
\frac{1}{x} & 1 & 0 & 0 \\
0 & \frac{1}{x} & 0 & 0 \\
0 & 0 & x & 1 \\
0 & 0 & 0 & x
             \end{smallmatrix}\right).
\end{align*}
This shows that the reduction of $t_1$ mod $p$ has order $p(p-1)$, and hence the same holds for its conjugate $t=\rho(T)$.

Now, since $-1 =\rho(S^2)$, it will suffice to show that the image of our subgroup in $\PSp_4(\FF_p)$ is the entire group. Let $G\subseteq \PSp_4(\FF_p)$ be this image.

The maximal subgroups of $\PSp_4(\FF_p)$ were worked out by Mitchell -- see Theorem 2.8 of \cite{King} for a description. In the notation of that Theorem, we can ignore cases (a) and (b). For the remaining cases, notice that $\langle t\rangle \cap \langle sts^{-1}\rangle = \{\pm 1\}$. Therefore, inside of $\PSp_4$, both $t$ and $sts^{-1}$ generate subgroups of size $p(p-1)/2$ with trivial intersection. It follows that $G$ has order satisfying $\abs{G} \geq p^2(p-1)^2/4$. This is enough to eliminate the cases (i) and (j) just by order considerations alone.

Next, we consider cases (g) and (h), where the maximal subgroups have order $p(p-1)^2(p+1)$, and, respectively, $p(p-1)(p+1)^2$, and the $p$-Sylow subgroups have order $p$.  We eliminate the possibility that $G$ is a subgroup of one of these two maximal subgroups in two steps: first, we show that in each case there are at most $(p+1)$ $p$-Sylow subgroups, and second, we find  $(p+2)$ conjugate subgroups of order $p$ inside $G$, and thus reach a contradiction. 

For the first step,  let $H$ be a $p$-Sylow subgroup; we may assume $H=\langle  t^{p-1} \rangle$.  Then the number of $p$-Sylow subgroups $n_p \equiv 1 \pmod{p}$, and equals the index of the normalizer $N(H)$ inside the group. Note that $\langle t \rangle$, which has order $p(p-1)$, centralizes $H$, and therefore normalizes it, so $p(p-1)$ divides the order $|N(H)|$.
 
In case (g), the size of the maximal subgroup is $p(p-1)^2(p+1)$, so $n_p$ is a divisor of $(p-1)(p+1)$, congruent to $1 \pmod{p}$. Therefore, $n_p=1$ or $n_p=p+1$. In case $(h)$, the group is isomorphic to a quotient of  $\U_2(\FF_{p^2}).2$ \cite{King}, and since $p \ne 2$, has the same number of $p$-Sylow subgroups as the unitary group $\U_2(\FF_{p^2})$. The $p$-Sylow subgroup is the unipotent part $U$ of its normalizer $N(U)$. At the same time, the image $L$ of   $\left\{ \begin{pmatrix} \lambda & 0 \\ 0 & \lambda^{-1} \end{pmatrix}: \lambda \in \FF_{p^2}\right\}$ under a homomorphism from $\SL_2(\FF_{p^2})$ to $\U_2(\FF_{p^2})$ normalizes $U$ (see \cite[Ch. 6-8]{Carter72}). Since $|UL|=p(p^2-1)$, the index $n_p \equiv 1 \pmod{p}$ is a divisor of $(p+1)$, and is again either $1$ or $p+1$.

Next, we explicitly construct $p+2$ conjugate subgroups of $G$ of order $p$.  Let  $H = \langle t^{p-1}\rangle$,  $K=\langle st^{p-1}s^{-1}\rangle$, and $K_d= t^{d(p-1)}Kt^{-d(p-1)}$.  Then $H$ acts on the set of $\{K_d\}_{d=1}^p$. By the orbit-stabilizer theorem, the orbit of $K_1$ either has size $1$ or $p$. Since $|K|=p$ would be a $p$-Sylow subgroup, its normalizer would equal the centralizer. However, we can explicitly check that 
\[t^{p-1}=
\left(\begin{smallmatrix} 
1 & 0 & -3\frac{x^2+1}{x^2-1} & -6 \frac{x^2+1}{(x^2-1)^2} \\
0 & 1 & 0  & \frac{x^2+1}{(x^2-1)^2} \\
0 & 0 & 1 &0 \\
0 & 0 & 0 & 1
\end{smallmatrix}\right)
 \]
 does not centralize $st^{p-1}s^{-1}$. Therefore, the orbit of $K$ must have size $p$, and we have $p$ conjugate subgroups of order $p$. At the same time, if $H=K_d$ for some $d$, then $\langle t^{p-1} \rangle =\langle t^{d(p-1)}s t^{p-1} s^{-1} t^{-d(p-1)} \rangle$, and conjugating by $t^{d(p-1)}$ gives $H= K$, which is not true since $H$ is upper triangular, while $K$ is lower triangular. So far, we have $(p+1)$ distinct subgroups of $G$ of order $p$.

Finally, for $k, d \in \{0, \dots, p-1\}$,   consider the matrices of the form $C_{d,k}\df t^{d(p-1)}s t^{k(p-1)}s^{-1} t^{-d(p-1)}$. Each $C_{d,k}$ belongs to $K_d$, and each matrix in any $K_d$ is of this form. Using the form of $t^{p-1}$ given above, we can explicitly check that the $(3, 2)$-entry of this matrix is $0$. On the other hand, the $(3,2)$-entry of $ s C_{1,1} s^{-1}$ is  $-18\frac{(x^2+1)^3}{(x^2-1)^4}$. These matrices are therefore never equal, since $x$ is a generator modulo $p$, and $p > 3$. For the same reason, $s C_{1,1} s^{-1}$ does not belong to $H$, since $H$ is upper triangular. Therefore, $\langle s C_{1,1}s^{-1}\rangle$ gives the $(p+2)$-nd distinct subgroup of order $p$ inside $G$, and we have  eliminated cases (g) and (h).

It remains to eliminate cases (c) through (f), and this can be done by a direct computation, since $G$ is finitely generated by $t$ and $s$. For (c), this amounts to verifying that $t$ and $s$ do not simultaneously stabilize a line and a hypersurface, which is easy to check. Therefore, all that remains is to treat cases (d), (e) and (f).

Case (d) concerns the stabilizer of a self-polar line in $\PP^3(\FF_p)$, which corresponds to a two-dimensional Lagrangian subspace of $\FF_p^4$. This is the case of the (conjugates of) the Siegel parabolic. It will thus suffice to verify that $R$ and $S$ do not act with simultaneous fixed points on $X(\FF_p)$, with notation as in Section \ref{ss:grassmanian}. When $p\equiv 3 \pmod{4}$, already $R$ acts without fixed points by Lemma \ref{l:rgrassmanian}, so we may suppose that $p\equiv 1\pmod{4}$. The fixed points of $S$ and $R$ are among the $A(a,b,c)$ and among the $B(a,b)$. Those $B(a,b)$ fixed by $S$ all have $b=0$ and $3a^2+x^2=0$, but $B(a,0)$ is fixed by $R$ if only if $x^4+3ax^2+3a^2 = 0$ and $a \neq 0$. This leads to $3a^2-3a+1=0$, and coupled with $x^2=-3a^2$, one deduces that $x$ has order $6$. This contradicts the fact that $x$ is a primitive root in $\FF_p$ for $p > 7$, thus no $B(a,b)$ is simultaneously fixed by $R$ and $S$.

Next we verify that no $A(a,b,c)$ is simultaneously fixed by both $R$ and $S$, where again we may assume that $p\equiv 1 \pmod{4}$. In order to be fixed by $S$, we must have $3a^2+bc \neq 0$, and then we must have that $x$ satisfies the equations:
\begin{align*}
a &= -\frac{a}{3a^2+bc}, & b&= -\frac{cx^2}{3a^2+bc}, & cx^2&= -\frac{b}{3a^2+bc}. 
\end{align*}
The first equation implies that either $a=0$ or $3a^2+bc=-1$. If $a=0$ then $bc \neq 0$ and the last two equations simplify to $b^2=-x^2$, $c^2=-x^{-2}$. Since $p\equiv 1\pmod{4}$, there are four solutions to this to consider: $A(0,\pm \sqrt{-1}x, \pm \sqrt{-1}x^{-1})$, there the signs need not match. It is easily checked that none of these four subspaces are stabilized by $R$.

It thus remains to treat the case when $3a^2+bc=-1$ and $a\neq 0$, in which case we have $b=cx^2$. This leads to considering $A(a,0,0)$ and $A(a,cx^2,c)$ where $3a^2+c^2x^2=-1$. It is easily seen that $A(a,0,0)$ is not fixed by $\rho(R)$, while for the final cases we observe that
\[
  \rho(R)A(a,cx^2,c) = A(a+c-1,-x^2(6a+3c-cx^2-2),x^{-2}(c-2))
\]
Hence, to be fixed, we deduce that $c=1$ and $x^2=-1$. This contradicts the fact that $x$ is primitive mod $p$, and so again we see that no $A(a,b,c)$ is simultaneously fixed by $R$ and $S$. This verifies that $G$ is not contained in any conjugate of a Siegel parabolic and eliminates case (d) from consideration.

Next we treat the case of (e), which is the stabilizer of a pair of skew polar lines in $\PP^3(\FF_p)$. Such a pair corresponds to a pair of Lagrangian subspaces with trivial intersection. If $R$ fixes such a pair then, since it is of order $3$, it fixes each Lagrangian subspace in the pair individually. Since we have already seen that $R$ and $S$ do not simultaneously fix a Lagrangian plane, to eliminate case (e), it suffices to show that $S$ does not exchange any pair of Lagrangian planes with trivial intersection that are each fixed by $R$. Again, to ensure that $R$ has fixed points, we may suppose that $p\equiv 1 \pmod{4}$.

In our analysis of the Siegel case (d), we saw that $R$ only fixes subspaces $A(a,b,c)$ with $3a^2+bc \neq 0$. Thus, we only need to consider pairs $\{A(a,b,c),A(d,e,f)\}$ where $3a^2+bc\neq 0$, $3d^2+ef \neq 0$ and
  \[
    3a^2+bc+3d^2+ef-6ad-ce-bf \neq 0,
  \]
  where this final condition ensures that $A(a,b,c) \oplus A(d,e,f) = \FF_p^4$. Now, the conditions $\rho(S)A(a,b,c) = A(d,e,f)$, $\rho(R)A(a,b,c) = A(a,b,c)$, and $\rho(R)A(d,e,f)=A(d,e,f)$ (along with the various nonvanishing conditions above) are algebraic equations, and one can show that they have no simultaneous solutions in a straightforward but tedious manner using the aid of a computer. Therefore, case (e) also does not arise.

It remains to treat case (f), which is the case of a regular spread. By \cite{Dye84}, this corresponds to the semidirect product $\PSL_2(\FF_{p^2}).\langle \sigma \rangle$, where $\sigma$ is an order two automorphism. We will examine the orders of elements in this group; since $p \ne 2$, we can safely work with $\SL_2(\FF_{p^2})$. By \cite[Table 1.1]{BC}, the elements in $\SL_2(\FF_{p^2})$ with order divisible by $p$ must either have order $p$ or $2p$. However, the order of $t=\rho_p(T)$ is $p(p-1)$. Since $p$ is large enough, $t$ is not contained in any subgroup of this group, and we have eliminated the last case. Therefore, the only possibility that remains is that $\langle \rho_p(T),\rho_p(S)\rangle = \Sp_4(\FF_p)$
\end{proof}

\begin{cor}
  \label{c:sp4surject}
  Let $p > 7$ be a prime, and let $x \in \ZZ$ be a primitive root mod $p$. Then the reduction $\phi_p$ induces a surjective map
  \[
  \widehat\Gamma \to \Sp_4(\Zp).
  \]
\end{cor}
\begin{proof}
This is a direct application of Theorem \ref{thm:irred}.
The symplectic form we are using behaves poorly if reduced modulo $3$, however, for $p> 7$ the group is isomorphic with the usual symplectic group $\Sp_{4}(\FF_p)$, and the necessary hypotheses required to apply Theorem \ref{thm:irred} hold by the standard theory of linear algebraic groups.
\end{proof}

Next, we examine the character $\chi$. Let $p \ne 2, 3$.
To bridge the results in \cite{Srin} with our explicit computations for $\chi$, consider the isomorphism of symplectic spaces $\mu: V \rightarrow V'$, where $(V, J)$ is our symplectic space with invariant form $J$,  $(V', J')$ is the symplectic space in \cite{Srin} with form  \[J'=\left(\begin{smallmatrix}
0 & 1 & 0 & 0 \\
-1 & 0& 0 & 0 \\
0 & 0 &0 & 1 \\
0 & 0 & -1 & 0
                 \end{smallmatrix}\right),\] and $\mu$ has matrix 

\begin{equation}
\label{eq:ismsymplectic}
M=\left(\begin{smallmatrix}1 & 0 & 0 & 0 \\
0 & 0& 0 & 1 \\
0 & 0 & 1/3 &0\\
0 & 1 & 0 &0
                 \end{smallmatrix}\right)
\end{equation}
 with respect to the standard basis.  Let $\theta_i$  be characters as defined in \cite{Srin}. 
\begin{lem}
  \label{l:chardecomp}
Let $p\ne 2, 3$. If $\chi$ is the character of the permutation representation of $\Sp_4(\FF_p)$ on $X(\FF_p)$, then one has
\[
\chi= 1 + \theta_9 +\theta_{11}.
\]
\end{lem}
\begin{proof}
Recall that the Siegel parabolic subgroups of $\Sp_4(\FF_p)$ stabilize the maximal isotropic subspaces; fixing such a Siegel parabolic $P$, we have that the permutation representation is the induced representation $\Ind_P^{\Sp_4}1$. By \cite[Table 2.8]{GM}, given the Borel $B \subseteq P \subseteq \Sp_4(\FF_p)$, the character of $\Ind_B^{\Sp_4}1$ decomposes as $1+2\theta_9+\theta_{11}+\theta_{12}+\theta_{13}$. By  Frobenius reciprocity and transitivity of induction, $\chi$ decomposes as the trivial representation, and  some linear combination of $\theta_9, \theta_{11}, \theta_{12}$ and $\theta_{13}$.

Next, we note that the degree of the permutation representation is $ \abs{X(\FF_p)}=p^3+p^2+p+1$. Using the degrees of the characters $\theta_i$ described in \cite{Srin}, $\chi$ decomposes as either $1+\theta_9+\theta_{11}$ or $1+\theta_9+\theta_{12}$.

Consider the matrix \[M=\left(\begin{smallmatrix} 1&0&0&1 \\ 0&1&0&0 \\ 0&-1/3 & 1& 0\\ 0&0&0&1\end{smallmatrix}\right) \in \Sp_4(\FF_p).\] Then one easily check that the fixed points of $M$ are $A(1/3, 0, c), A(-1/3, 0, c)$ for all $c \in \FF_p$, and $B(0,0)$. Therefore, $M$ has $2p+1$ fixed points. One can also check that via the isomorphism $\mu$ defined in Equation (\ref{eq:ismsymplectic}), $M$ is conjugate to the class $A_{31}$ in \cite{Srin}. Examining the tables in \cite{Srin},  the value of $(1+\theta_9+\theta_{11})$  on $A_{31}$ is $2p+1$, and the value of $(1+\theta_9+\theta_{12})$ is $p+1$, so $\chi$ must decompose as $1+\theta_9+\theta_{11}$.
\end{proof}

Based on this decomposition and the tables in \cite{Srin}, we summarize in the following table  the values of $\chi$  on certain conjugacy classes in $\Sp_4(V)$ that will be useful for cusp calculations later on.

\begin{table}[ht]
{\renewcommand{\arraystretch}{1.2}
\begin{tabular}{c|c } 
 Class &  Value of $\chi$ \\ 
  \hline
 $A_1$, $A_1'$ & $p^3+p^2+p+1$  \\ 
  %\hline
$A_{31}$, $A_{31}'$ & $2p+1$   \\ 
  %\hline
$B_8'$ & $p+3$  \\ 
  %\hline
$B_9$ & $3$ 
\end{tabular}}
\caption{Values of $\chi$ on certain conjugacy classes in $\Sp_4(V)$}
\label{t:charvals}
\end{table}

\begin{prop}
\label{p:cusps}
  Suppose that $p$, $x$ and $y$ satisfy the hypotheses of Theorem \ref{t:sp4surject}. Then
  \[
c=2p+12.
  \]
\end{prop}
\begin{proof}

 In light of Lemma \ref{l:cusps}, we need to consider $T^d$ for $d | p(p-1)$. In particular, either $d$ is a divisor of $p-1$, or $d=p\cdot d'$ where $d'$ is a divisor of $p-1$. 
As in the proof of Theorem \ref{t:sp4surject}, note that $T$ is conjugate by $P$ to $t_1$. We calculate  that \[ t_1^d=\left(\begin{smallmatrix}
x^{-d} & dx^{1-d} & 0 & 0 \\
0 & x^{-d}& 0 & 0 \\
0 & 0 & x^d & dx^{d-1} \\
0 & 0 & 0 & x^{d}
                 \end{smallmatrix}\right).\] 

For $1 \le d < \frac{p-1}{2}$, we have that $x^d \neq x^{-d}$. In these cases, the minimal polynomial (in variable $u$) of $T^d$ is $(u-x^d)^2(u-x^{-d})^2$, and therefore $T^d$ belongs to the conjugacy class $B_9(i)$ as in \cite{Srin}. By Table \ref{t:charvals}, in this case $\chi(T^d)=3$.

For  $p-1 < d< \frac{p(p-1)}{2}$, where in particular $d=p\cdot d'$, again $x^d \neq x^{-d}$. Then $T^d$ has minimal polynomial $(u-x^d)(u-x^{-d})$, and therefore belongs to $B_8(i)$. In such cases, $\chi(T^d)=p+3$. 

Clearly, $T^{p(p-1)}$ belongs to $A_1$ and $T^{\frac{p(p-1)}{2}}$ belongs to $A_1'$, for which $\chi(T^d)=p^3+p^2+p+1$. We have left to find the classes to which $T^d$ belongs for $d=\frac{p-1}{2}$ and $d=p-1$.

Let $d=\frac{p-1}{2}$. Using the matrices $P$ and $t_1^d$, we have 
\[T^d=\left(\begin{smallmatrix}
-1 & 0 & \frac{3(x^2+1)}{2(x^2-1)} & \frac{3(x^2+1)}{(x^2-1)^2} \\
0 & -1& 0 & -\frac{x^2+1}{2(x^2-1)} \\
0 & 0 & -1& 0 \\
0 & 0 & 0 &-1
                 \end{smallmatrix}\right).\]
 Straightforward computations show that via the isomorphism $\mu$ in Equation (\ref{eq:ismsymplectic}), $T^d$ is conjugate to the class $A_{31}'$ by the matrix 

\[\left(\begin{smallmatrix}
-\frac{x^2-11}{4(x^2-1)} & 0& \frac{x^2+13}{4(x^2-1)}&0\\
 0& -\frac{2(x^2-1)}{x^2+1} & 0 & \frac{2(x^2-1)}{x^2+1}\\
0 & \frac{x^2+13}{2(x^2+1)}  & 0 & \frac{x^2-11}{2(x^2+1)}\\
-1 & 0 & -1 &0
                 \end{smallmatrix}\right).\]
 Since $x$ is a multiplicative generator modulo $p$, $x^2 \ne \pm 1 \pmod{p}$ and  one can check that this matrix belongs to $\Sp_4(V')$. Therefore, $\chi(T^{\frac{p-1}{2}})=2p+1$.

Likewise, for $d=p-1$, we have 
\[T^{p-1}=\left(\begin{smallmatrix}
1 & 0 & -\frac{3(x^2+1)}{(x^2-1)} & -\frac{6(x^2+1)}{(x^2-1)^2} \\
0 & 1& 0 & \frac{x^2+1}{x^2-1} \\
0 & 0 & 1& 0 \\
0 & 0 & 0 &1
                 \end{smallmatrix}\right),\]
 which is conjugate to $A_{31}$   via $\mu$  by the matrix 
\[\left(\begin{smallmatrix}
-\frac{x^2-5}{2(x^2-1)} & 0& \frac{x^2+7}{2(x^2-1)}&0\\
 0& -\frac{x^2-1}{x^2+1} & 0 & \frac{x^2-1}{x^2+1}\\
0 & \frac{x^2+7}{2(x^2+1)}  & 0 & \frac{x^2-5}{2(x^2+1)}\\
-1 & 0 & -1 &0
                 \end{smallmatrix}\right),\] which again belongs to $\Sp_4(V')$. Therefore, $\chi(T^{p-1})=2p+1$.

Finally, using Lemma \ref{l:cusps}, we have

\begin{align*}
p(p-1) c= &  \sum_{\substack{d| p-1 \\ 1< d < \frac{p-1}{2}}} 3\varphi\left(\frac{p(p-1)}{d}\right) + \sum_{d \in\{ \frac{p-1}{2}, \, p-1\}} (2p+1)\varphi\left(\frac{p(p-1)}{d}\right) +\\
 &  \sum_{\substack{d| p(p-1) \\ p-1< d < \frac{p(p-1)}{2}}} (p+3)\varphi\left(\frac{p(p-1)}{d}\right) + \sum_{d \in\{ \frac{p(p-1)}{2}, \,  p( p-1) \}} (p^3+p^2+p+1)\varphi\left(\frac{p(p-1)}{d}\right).
\end{align*}
Rearranging the terms,  using $\varphi(p)=\varphi(2p)=p-1$ for odd $p$, and the identity $\sum_{d|n} \varphi(n/d)=n$, we get 
\begin{align*}
p(p-1) c & = 3 \sum_{d|p(p-1)} \varphi\left(\frac{p(p-1)}{d}\right)+ 2(2p-2)(p-1)
 + p \sum_{d| (p-1)}\varphi\left(\frac{p-1}{d}\right) + 2 (p^3+p^2-2)\\
& = 3p(p-1)+4(p-1)^2+p(p-1)+2(p^3+p^2-2)\\
& =p(p-1)(2p+12),
\end{align*}
and the result follows.
\end{proof}

We can say more about the cusps in the following corollary.

\begin{cor}
There are $3$ cusps of width $1$, $4$ cusps of width $\frac{p-1}{2}$, one cusp of width $p$, and $2p+4$ cusps of width $\frac{p(p-1)}{2}$.
\end{cor}

\begin{proof}
We use the values $\chi(T^d)$ from the proof of Proposition \ref{p:cusps}. First, we have $c_1=\chi(T)=3$. 

Next, consider a divisor $d \mid \frac{p(p-1)}{2}$ such that $1 < d < \frac{p-1}{2}$. Then

\[ c_d=\frac{1}{d}\left(\chi(T^d) -\sum_{\substack{f|d \\ 1 \le f<d}} f c_f\right)=\frac{1}{d}\left(3-3-\sum_{\substack{f|d \\ 1 < f<d}} f c_f\right) \ge 0,\]
therefore $c_d=0$.
This implies that 
\begin{align*}
 c_{\frac{p-1}{2}}& =  \frac{2}{p-1} \left( \chi(T^{\frac{p-1}{2}}) -\sum_{\substack{f\mid \frac{p-1}{2} \\ 1 \le f < \frac{p-1}{2} }}fc_f \right) =  \frac{2}{p-1} \left( 2p-2-\sum_{\substack{f\mid \frac{p-1}{2} \\ 1 <f < \frac{p-1}{2} } }fc_f \right) =4.
\end{align*}

Next, we have
\[c_p=\frac{1}{p}\left( \chi(T^p)-c_1\right)=1.\]
On the other hand, suppose $d=p \cdot d'$ with $d'$ a divisor of $\frac{p-1}{2}$ such that $1 < d' < \frac{p-1}{2}$. Then
\[ c_d=\frac{1}{d}\left(\chi(T^d) -\sum_{\substack{f|d \\ 1 \le f<d}} f c_f\right)=\frac{1}{d}\left(p+3-3-p-\sum_{\substack{f|d \\ f \ne 1, p, d}} f c_f\right) \ge 0,\]
so again $c_d=0$.

The only other divisor  left is $\frac{p(p-1)}{2}$ itself, and
\[c_{\frac{p(p-1)}{2}}=2p+12-1-3-4=2p+4. \]
\end{proof}

Before we can calculate the genus of the curve corresponding to our noncongruence group, we need one last result on the number of elliptic points of order $3$.

\begin{lem}
  \label{l:rgrassmanian}
 Let $p > 3$, and let $\veps_3$ denote the number of fixed points of $\rho(R)$ in its action on the symplectic Grassmanian $X(\FF_p)$. Then
  \[
  \veps_3 = p+1+(p+1) \left( \frac{-3}{p} \right)= \begin{cases}
    2p+2 & p\equiv 1\pmod{3},\\
    0 & p\equiv 2\pmod{3}.
  \end{cases}
\]
In particular, $\veps_3$ is independent of $x$ and $y$.
\end{lem}
\begin{proof}
We use the character $\chi=1+\theta_9+\theta_{11}$ defined earlier. In particular, $\varepsilon_3=\chi(R)$.
Note that the minimal polynomial (in variable $u$) of $TS$, which is conjugate to $R$ in $\Sp_4(V)$, is $u^3+1$. Since $\gamma^2 \neq \pm 1$ in $\FF_p$ for generators $\gamma$ of $\FF_p$ or $\FF_{p^2}$ , a straightforward computation shows that the only conjugacy classes in \cite{Srin} with minimal polynomial of degree $3$ are  $C_1(i), C_1'(i), C_3(i), C_3'(i)$, and $D_{21}, D_{22}, D_{23}, D_{24}$. We can exclude the classes $D_{21}, D_{22}, D_{23}, D_{24}$, whose minimal polynomials are fixed and do not equal $u^3+1$.

Note that the minimal polynomials of classes $C_1(i)$, $C_1'(i)$, $C_3(i)$ and $C_3'(i)$ have the general form $$u^3-\frac{\gamma^{2i}\pm \gamma^i+1}{\gamma}u^2\pm\frac{\gamma^{2i}\pm \gamma^i +1}{\gamma} \mp 1,$$ where $\gamma$ is either a generator for $\FF_{p^2}$ in case of conjugacy class $C_1$, or a generator for $\FF_p$ in case of class $C_3$. Clearly $R$ must belong to one of these conjugacy classes, and the minimal polynomial of at least one class must be $u^3+1$. Since $p \ne 2,3$, we have two cases. If $p \equiv 1 \pmod{3}$, the polynomial $v^2 \pm v+1$ has a root in $\FF_p$, and for no value of $\gamma$ as a generator of $\FF_{p^2}$ and $i$ as specified in \cite{Srin} is $\gamma^{2i}\pm \gamma^i +1 =0$. Therefore, $R$ must be in some class $C_3(i)$ or $C_3'(i)$, for which $\chi(C_3(i))=\chi(C_3'(i))=2p+2$.  If $p \equiv 2 \pmod{3}$, the root of the polynomial $v^2 \pm v+1$ must instead lie in $\FF_{p^2}$, so $R$ must be in some class $C_1(i)$ or $C_1'(i)$, for which $\chi(C_1(i))=\chi(C_1'(i))=0$. One can easily verify that these formulas agree with $p+1+(p+1) \left( \frac{-3}{p}\right)$ when $p > 3$. 
\end{proof}

\begin{thm}
  \label{t:genus}
  Suppose that $p$, $x$ and $y$ satisfy the hypotheses of Theorem \ref{t:sp4surject} and let $g$ be the genus of the curve defined by a point stabilizer for $\Gamma$ in its action on the symplectic Grassmanian $X(\FF_p)$ for a prime $p > 7$. Then
  \[
  g = \begin{cases}
    \frac{p^3 + p^2 - 22p - 76}{12} & p\equiv 1\pmod{12},\\
    \frac{p^3 + p^2 - 14p - 68}{12} & p\equiv 5\pmod{12},\\
 \frac{p^3 + p^2 - 22p - 70}{12} & p\equiv 7\pmod{12},\\
 \frac{p^3 + p^2 - 14p - 62}{12} & p\equiv 11\pmod{12}.
  \end{cases}
\]
\end{thm}
\begin{proof}
Since the group has index $(p^2+1)(p+1)$ in $\SL_2(\ZZ)$, we use the formula $$g=1+\frac{(p^2+1)(p+1)}{12}-\frac{\varepsilon_2}{4}-\frac{\varepsilon_3}{3}-\frac{c}{2}.$$ 
Putting together the formulas for $\varepsilon_2, \varepsilon_3$ and $c$ obtained in Lemmas   \ref{l:sgrassmanian},  \ref{l:cusps} and \ref{l:rgrassmanian}, we get 
$$ g = \frac{1}{12} \left( p^3+p^2-\left(18+4 \left( \frac{-3}{p}\right) \right) p -\left(69+3\left( \frac{-1}{p} \right) +4 \left( \frac{-3}{p} \right) \right) \right),$$ and  obtain the result by matching up with the four possible cases for $p$.
\end{proof}

The genus of the curves discussed in Theorem \ref{t:genus} are listed in Table \ref{tab:genus} for the first several primes.

\begin{table}[ht]
{\renewcommand{\arraystretch}{1.2}
\begin{tabular}{c|c } 
 Prime $p$&  Genus \\ 
  \hline
11 & 103  \\ 
  %\hline
13 & 167  \\ 
  %\hline
17 & 408 \\ 
  %\hline
19 &561 \\ 
  %\hline
23 & 1026 \\ 
  %\hline
29 & 2063 \\ 
  %\hline
31 &2500
\end{tabular}}
\caption{Genus values for small primes.}
\label{tab:genus}
\end{table}

\begin{rmk}
From the genus formula in Theorem \ref{t:genus}, we can see that for weights $k\ge 2$, the dimension of the corresponding spaces of modular forms grows on the  order $O(kp^3)$.
\end{rmk}

Finally, let us conclude this discussion by confirming that the groups considered above are indeed noncongruence.
\begin{thm}
Let $p$, $x$ and $y$ satisfy the hypotheses of Theorem \ref{t:sp4surject}. Let $H_p\subseteq \Gamma$ denote the inverse image under $\phi_p$ of a point stabilizer for the action of $\Sp_4(\FF_p)$ on the symplectic Grassmanian $X(\FF_p)$. Then $H_p$ is a noncongruence subgroup.
\end{thm}
\begin{proof}
We have seen that since $\phi_p(T)$ has order $p(p-1)$, due to our hypotheses, the Wohlfahrt level is $p(p-1)$. Thus, if $H_p$ is congruence, we must have $\Gamma(p(p-1)) \subseteq H_p$. Since $\ker \phi_p$ is the intersection of the conjugates of $H_p$, this would imply that $\Gamma(p(p-1)) \subseteq \ker \phi_p$. But this is a contradiction as one sees by considering the indices of $\Gamma(p(p-1))$ and $\ker \phi_p$, which grow on the order of $p^6$ and $p^{10}$, respectively.
\end{proof}

\begin{rmk}
At this time we do not know how these groups behave as one fixes the prime $p$ but varies the primitive root $x$ modulo $p$. We can at least say that the groups are not necessarily conjugate. For example, for $p=11$ and $x=2$, we have $ST^{20}S^{-1}T^{33}S^{-1}T^{20}S^{-1}T^{33} \in \ker \phi_p$, but this is not true if $x=8$. Thus, the kernels in these two cases are distinct, and so a point stabilizer when $x=2$ cannot be conjugate to a point stabilizer when $x=8$. On the other hand, we can still ask the following: does one obtain groups that are equivalent under the action of the absolute Galois group $\Gal(\Qbar/\QQ)$ acting via outer automorphisms on the finite-index subgroups of $\Gamma$?
\end{rmk}

%Bibliography:
\bibliographystyle{plain}
\bibliography{Revised_families_of_subgroups}
\end{document}